\documentclass[reqno,12pt]{amsart}
\usepackage{DocumentPreamble}

\AddTitle{Finding all cospectral mates over a number field}
\AddAuthor{Alexander Van Werde}
\email{\href{mailto:a.van.werde@uni-muenster.de}{a.van.werde@uni-muenster.de}}
\address{University of Münster, Germany}
 
\AddMSC{05C50}
\AddMSC{15B36}
\AddMSC{11C20}
\AddMSC{11R04}
\AddMSC{05C60}

\AddKeyword{Integer matrix}
\AddKeyword{Dedekind domain}
\AddKeyword{discriminant}
\AddKeyword{sufficient condition}
\AddKeyword{Krylov subspace}
\AddKeyword{computational algebraic number theory}
\AddKeyword{spectral graph theory}

\begin{document} 
\begin{abstract}   
    We investigate a notion of cospectrality for integer matrices that is parameterized by algebraic number fields.
    Given a number field and a symmetric integer matrix, we wonder when conjugating the integer matrix by an orthogonal matrix with entries in the given field can produce new integer matrices.

    Our results concern sufficient conditions for the associated notion of spectral determination, and we give constraints on the orthogonal matrices when the conditions are not applicable. 
    The results use the discriminant of the characteristic polynomial and properties of Krylov subspaces. 
    We leverage the theory to develop an algorithm to find all cospectral mates over a given (small) field. 
    An implementation of the algorithm is made available.
\end{abstract}

\maketitle  

\begin{minipage}{0.9\textwidth}
\let\clearpage\relax
{\scriptsize
\tableofcontents 
} 
\end{minipage} 
\newpage

\section{Introduction}
Eigenvalues of symmetric integer matrices frequently reveal rich combinatorial information. 
For example, spectral methods are often applied in the study of graphs and their signed or weighted variants \cite{belardo2019open,mohar1993eigenvalues,brouwer2011spectra,godsil2013algebraic}.  
A classical question in this area asks when the spectrum gives complete information.

More precisely, consider a symmetric integer matrix $\bX\in \bbZ^{n\times n}$ and let $\phi_{\bX}(x) \de \det(x\bI -\bX)$ be its characteristic polynomial.
Then, a symmetric matrix $\bY\in \bbZ^{n\times n}$ is a \emph{cospectral mate} of $\bX$ if $\phi_{\bX} = \phi_{\bY}$. 
We say that $\bX$ is \emph{determined by its spectrum} if the only cospectral mates are the trivial ones $\bY = \bS^{\T} \bX \bS$ following from conjugation of $\bX$ by a signed permutation $\bS$. 
Here, a \emph{signed permutation} is a matrix with a single nonzero entry $\pm 1$ in every row and column.

If $\bX$ and $\bY$ are adjacency matrices of graphs and thus have nonnegative entries, then the signs may be removed and one recovers the notion of graph isomorphism. 
For signed graphs, conjugation by signed permutation corresponds to \emph{switching equivalence} \cite{belardo2019open}.
Questions concerning adjacency cospectrality of graphs date back to the 1956 work of G\"unthard and Primas on quantum chemistry \cite{gunthard1956zusammenhang}. 
The topic can also be viewed as a discrete variant on the question ``Can one hear the shape of a drum?'' that was popularized by Kac in 1966 \cite{kac1966can}. 

There is unfortunately no known general method to decide whether a matrix is determined by spectrum, or to find its cospectral mates.  
(That is, aside from exhaustive enumeration which quickly becomes intractable \cite{brouwer2009cospectral,haemers2004enumeration}.) 
In many cases, however, an improved understanding has been achieved through the analysis of more restrictive notions such as cospectrality through small \emph{switching methods} \cite{godsil1982constructing,abiad2026counting,schwenk1973almost} or \emph{generalized cospectrality}  \cite{johnson1980note,wang2006sufficient,qiu2023generalized,wang2026graph,wang2017simple}.
Note that $\bX$ and $\bY$ are cospectral mates if and only if there exists an orthogonal matrix $\bQ$ with $\bQ^{\T}\bX \bQ = \bY$. 
The restricted notions can be interpreted as putting additional constraints on the orthogonal matrix. 

Switching methods restrict to orthogonal matrices of the form $\operatorname{diag}(\bQ_{m}, \bI_{n-m})$ with a small nontrivial block of size $m\ll n$; see the unified framework of Abiad, Van de Berg, and Simoens \cite[Definition 1]{abiad2026counting}.
Restricting the block size $m$ can be interpreted as focusing on local obstructions for spectral characterization. 
Classical examples are \emph{Godsil--McKay switching} \cite{godsil1982constructing} and \emph{Schwenk switching} \cite{schwenk1973almost}. 
An efficient algorithmic implementation of switching methods will appear in a forthcoming work of Simoens and Van Overberghe \cite{simoens2026prep}. 

Generalized cospectrality restricts the orthogonal matrix $\bQ$ to satisfy $\bQ \zeta = \eta$ for given integer vectors $\zeta, \eta \in \bbZ^n$; see Wang et~al.~\cite{wang2006sufficient,qiu2023generalized,wang2026graph,wang2017simple}. 
Notably, sufficient conditions that can certify that a matrix is determined by its generalized spectrum have been established \cite{wang2006sufficient,qiu2023generalized,wang2026graph,wang2017simple}.
Such conditions are conjectured to apply with non-vanishing frequency to high-dimensional matrices  \cite{lvov2026satisfaction,wang2017simple}.

Moreover, algorithms to find all generalized cospectral mates have recently been developed by Wang and Wang \cite{wang2025haemers}.
They concern the case when the orthogonal matrix has to fix the all-ones vector. 
That is, $\bQ e = e$ with $e=(1,\ldots,1)^{\T}$.
This notion has particular significance, as there are a number of equivalent formulations dating back to the 1980 work of Johnson and Newman \cite{johnson1980note}. 

\pagebreak[4]
Generalized cospectrality has the pleasing feature that it does not require any a-priori bound on the size of obstructions to spectral characterization.  
However, while switching methods naturally give a hierarchy of relaxations with $m = n$ recovering the classical notion of cospectrality, generalized cospectrality is a legitimately restricted notion that does not necessarily recover the classical variant.
 
The idea of the present paper is as follows: while it may not always be possible to relate cospectral $\bX$ and $\bY$ by an orthogonal matrix $\bQ$ satisfying $\bQ \zeta = \eta$ for given integer vectors $\zeta, \eta\in \bbZ^n$, we can certainly assume that the entries of $\bQ$ come from an algebraic number field. 
Indeed, the latter can be verified by using the orthogonal matrices from the eigendecomposition. 
This motivates studying the notion that arises by constraining the field:

\begin{definition}\label{def: CospectralOverK}
    Fix some number field $K$. 
    Two symmetric integer matrices $\bX,\bY \in \bbZ^{n\times n}$ are \emph{cospectral mates over $K$} if there exists some $\bQ \in K^{n\times n}$ that is orthogonal $\bQ^{\T} \bQ = \bI$ and satisfies $\bQ^{\T} \bX \bQ = \bY$.   
\end{definition}

The simplest case of $K = \bbQ$ is already of significant interest.
Indeed, this case essentially recovers generalized cospectrality, as the constraint $\bQ \zeta = \eta$ for integer $\zeta,\eta\in \bbZ^n$ implies that $\bQ$ has rational entries under mild conditions that are necessary to make the theory a useful relaxation; see \eg \cite[Lemma 2.4]{wang2006sufficient}, \cite[Lemma 2.1]{qiu2023generalized}, \cite[Theorem 1.3]{wang2026graph}.
Cospectrality over the rationals as an extension of generalized cospectrality was notably emphasized in a 2016 preprint of Wang and Yu \cite{wang2016square} who proved sufficient conditions.
See also Bhargava, Gross, and Wang \cite[Proposition 35]{bhargava2017positive} for related conditions with pairs of $p$-adic matrices.

Section \ref{sec: Results} presents our results. 
We give a condition to rule out cospectrality over an arbitrary number field $K$, and establish constraints on the potential orthogonal matrices $\bQ\in K^{n\times n}$ when the condition is not applicable. 
Moreover, we leverage the results to develop an algorithm that can find all cospectral mates of $\bX$ over $K$.  
The sufficient condition recovers \cite[Theorem 1.4]{wang2016square} when $K = \bbQ$.   
To our knowledge, the constraints on the orthogonal matrices were not previously known for $K=\bbQ$, nor were algorithms known without further restrictions.

The only assumption that is strictly required for our theory is that $\bX$ should not have repeated eigenvalues. 
In particular, this assumption is applicable with high probability to the adjacency matrix of an Erd\H{o}s--R\'enyi random graph with constant edge probability \cite{tao2017random}.
For practical applications of the algorithms, some additional assumptions are necessary to make computations tractable; see Section \ref{sec: AlgorithmicSearch}.  
Most crucially, we require that the field $K$ is not too large. 
This requirement makes Definition \ref{def: CospectralOverK} into a meaningful relaxation.

We have implemented the algorithm building on methods from \texttt{PARI} \cite{PARI2} and \texttt{SageMath} \cite{sagemath} for computational algebraic number theory. 
The source code is distributed in a repository named \texttt{COSPECTRAL}:
\begin{center}
   \url{https://github.com/Alexander-Van-Werde/cospectral}
\end{center}
Section \ref{sec: Numerical} presents numerical results applying the algorithm to random integer matrices in dimension $\leq 100$.
We find that $K=\bbQ$ is a frequent source of cospectral mates, and a nontrivial fraction of additional mates arise over larger fields like the quadratic fields $K=\bbQ(\sqrt{d})$. 
 
The idea to parametrize cospectrality by number fields arose from recent work of the author on the probability that the conjugation action of a fixed orthogonal matrix on a random integral matrix again yields an integral matrix; see Van Werde \cite{vanwerde2026exact}. 
The probabilistic theory in \cite{vanwerde2026exact} worked equally well over the rings of integers of number fields as over $\bbZ$, raising the question whether sufficient conditions could also be extended. 

Some special cases of number fields have also appeared in sufficient conditions for generalized spectral characterization of \emph{mixed graphs} \cite{wang2019generalized,ji2025mixed}.
That is, graphs with both directed and undirected edges.
The relevance of the number fields was there mainly to indicate edge directions.
For instance, Wang, Qiu and Qian \cite{wang2019generalized} considered a \emph{Hermitian adjacency matrix} with entries from $\{0,1, i,-i \}$ with $i=\sqrt{-1}$.  
They studied a notion of cospectrality that effectively restricts attention to conjugation by unitary matrices $\bU$ with entries in the Gaussian rationals $\bbQ(i)$ satisfying $\bU e = e$ with $e$ the all-ones vector \cite[Theorem 10]{wang2019generalized}.

We conclude this introduction with a few illustrative examples for Definition \ref{def: CospectralOverK}, and by outlining the structure of the remainder of the paper.

\begin{example}\label{ex: CospectralLoops}
    A minimal example of a pair of adjacency cospectral graphs with loops is visualized in Figure \ref{fig:graphs_matrices}.
    The characteristic polynomial is $x^4 -2x^3 -x^2 +2x$ which splits over $\bbQ$ with roots $2$, $\pm 1$,
    and $0$.   
    However, it turns out that these graphs are \emph{not} adjacency cospectral over the rationals. 
    
    To explain this, note that normalizing eigenvectors here requires a field extension. 
    For instance, the eigenvectors of $\bA_1$ include $(0,1,1,1)^{\T}$ and $(0,0,-1,1)^{\T}$. 
    Normalizing thus requires dividing by $\sqrt{3}$ and $\sqrt{2}$. 
    In fact, one can check that all eigenvectors of $\bA_1$ and $\bA_2$ can be normalized over $\bbQ(\sqrt{2}, \sqrt{3})$, so the graphs are certainly cospectral over the latter field. 
    Indeed, one can then find $\bQ_1,\bQ_2$ with
    \begin{equation}
        \bQ_1^{\T} \bA_1 \bQ_1 = \operatorname{diag}(2,1,-1, 0) =  \bQ_2^{\T} \bA_2 \bQ_2,   
    \end{equation}
    and hence $\bQ^{\T} \bA_1\bQ =  \bA_2 $ with $\bQ\de \bQ_1 \bQ_2^{\T}$.      
    Actually, the matrices are already cospectral over $\bbQ(\sqrt{2})$:
     { 
        \setlength{\arraycolsep}{4pt}
        \renewcommand{\arraystretch}{0.9}
    \begin{equation}
        \bQ^{\T} \bA_1 \bQ = \bA_2\ \ \textnormal{ with }\ \ \bQ = \frac{1}{\sqrt{2}} \small \begin{pmatrix}
             0 & 0 & 1 & -1 \\ 
             0 & 0 & 1 & 1\\ 
             1 & 1 & 0 & 0 \\ 
             -1 & 1 & 0 & 0
        \end{pmatrix}
    \end{equation}
    In this example, cancellations lead to cospectrality over a smaller field than is necessary to define the normalized eigenvectors of the individual matrices.
    }
\end{example} 
\begin{figure}[htbp]
    \centering
    \begin{tabular}{ 
        >{\centering\arraybackslash}m{0.17\textwidth} 
        >{\centering\arraybackslash}m{0.32\textwidth} 
        >{\centering\arraybackslash}m{0.17\textwidth} 
        >{\centering\arraybackslash}m{0.32\textwidth} 
    }
        \includegraphics[height=4.5em]{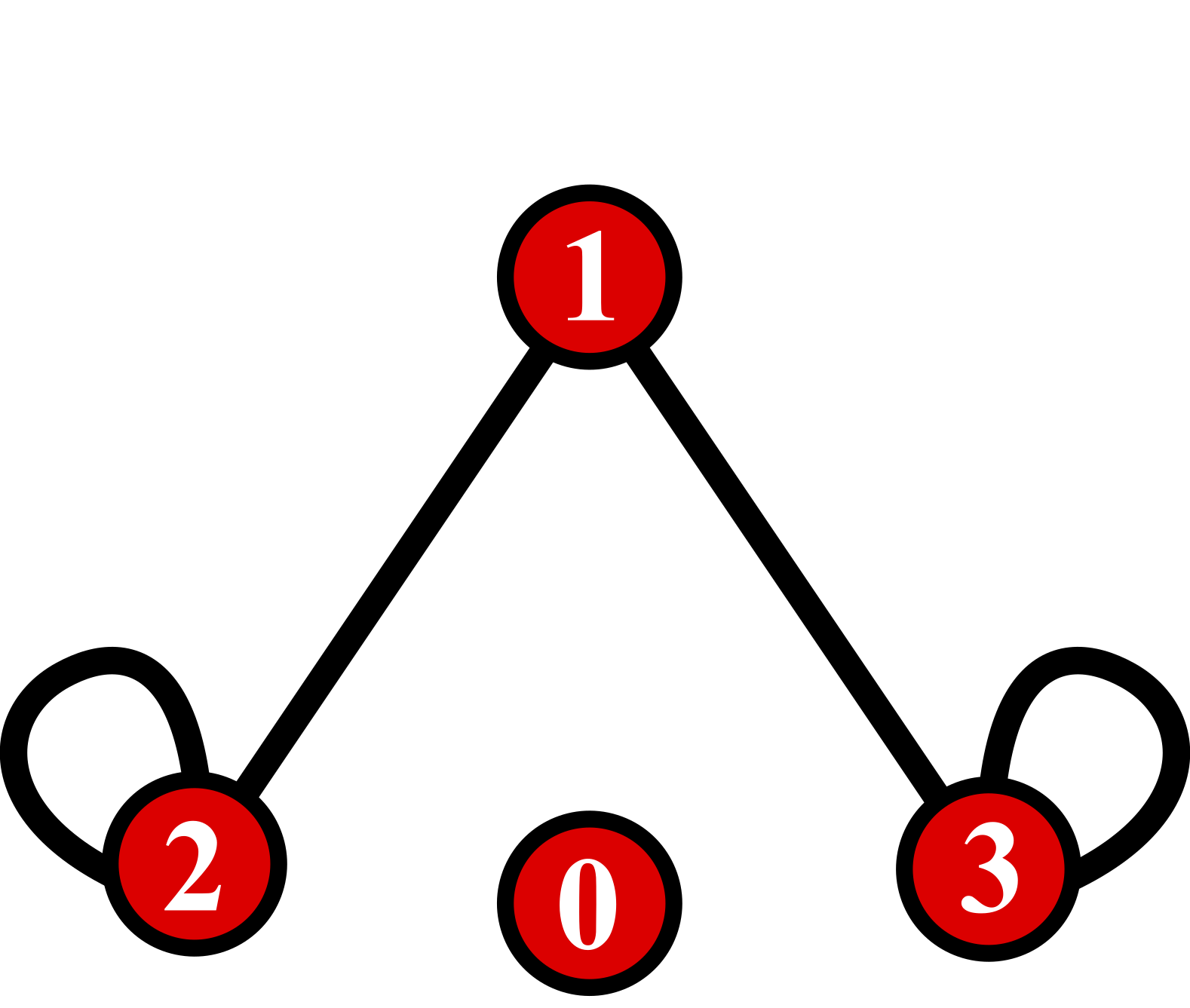} & { 
        \setlength{\arraycolsep}{4pt}
        \renewcommand{\arraystretch}{0.9}
        $\bA_1 = \begin{pmatrix}
        0 & 0 &0 &0\\ 
        0& 0& 1 & 1\\ 
        0 & 1 & 1 & 0\\ 
        0 & 1 & 0 & 1
        \end{pmatrix}$}&
        \includegraphics[height=4.5em]{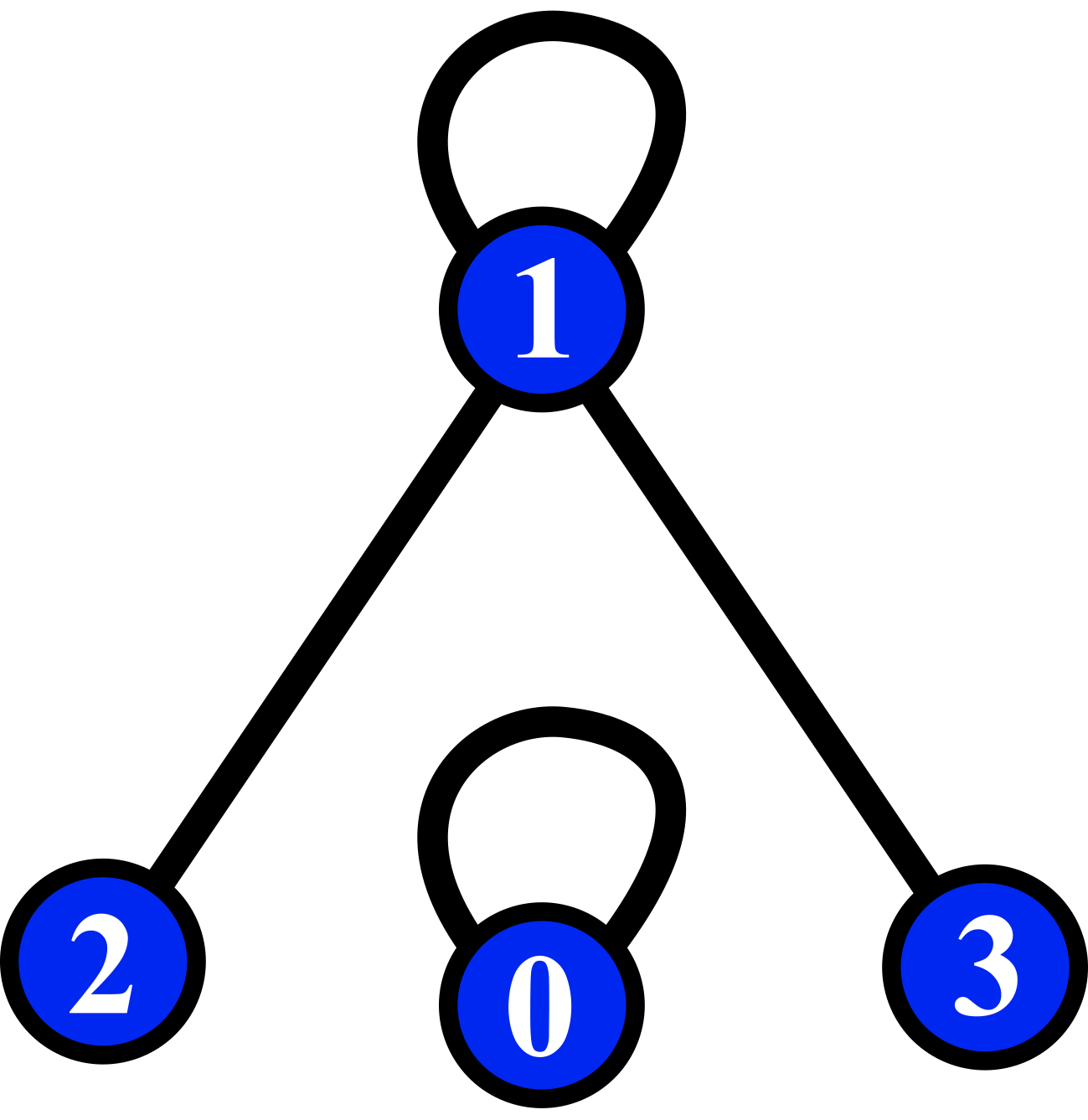} & { 
        \setlength{\arraycolsep}{4pt}
        \renewcommand{\arraystretch}{0.9}
        $\bA_2 = \begin{pmatrix}
        1 & 0 &0 &0\\ 
        0& 1& 1 & 1\\ 
        0 & 1 & 0 & 0\\ 
        0 & 1 & 0 & 0
        \end{pmatrix}$}
    \end{tabular}
    \caption{Two cospectral graphs and their adjacency matrices.}
    \label{fig:graphs_matrices}
\end{figure}

\pagebreak[4]
\begin{example}\label{ex: CospectralInteger}
    The theory and algorithms are particularly relevant for matrices with large entries or high dimensionality, since exhaustive enumeration of candidate cospectral mates then becomes infeasible.  

    In high dimensional settings, \emph{Haemers' conjecture} proposes that the probability of cospectrality of dense Erd\H{o}s--R\'enyi random graphs should tend to zero as the number of nodes grows large \cite{van2003graphs,haemers2016almost}.
    In fixed dimensionality, on the other hand, a random symmetric matrix with entries uniform from $\{1,\ldots,M \}$ is cospectral with probability bounded away from zero as $M\to \infty$. 
    (See, \eg \cite{vanwerde2026exact}.) 
    It is not necessarily easy to guess whether the following matrix will be cospectral: 
    {        \setlength{\arraycolsep}{4pt}
        \renewcommand{\arraystretch}{0.9}
    \begin{equation}
        \bX = 
        \begin{pmatrix}
            57& 21& 80& 57& 19\\
            21& 56& 14&  5& 78\\
           80& 14& 82& 71& 43\\
            57&  5& 71& 17& 46\\
           19& 78& 43& 46& 14
        \end{pmatrix}.
        \label{eq:DangerousNet}
    \end{equation}
    Our algorithm reveals that $\bX$ is cospectral over $\bbQ$. 
    We have $\bQ^{\T} \bX \bQ = \bY$ with 
    \begin{equation}
        \bY  = 
        \begin{pmatrix}
            33&   1& 115&  49&   6\\
            1&  -7 & 23&  12 & -9\\
            115&  23& 136&  14&  30\\
             49&  12&  14&  90&  29\\
             6&  -9&  30&  29& -26 
        \end{pmatrix}, \quad 
        \bQ = 
        \frac{1}{31}\begin{pmatrix}
             19& -20&  10&  -6&   8\\
             -6& -10&   5&  28&   4\\
         -4&  14&  24&  -2&  13\\
          8&   3&  14&   4& -26 \\
        22&  16&  -8&  11&   6
        \end{pmatrix}.
    \end{equation}
    Note that $\bQ e \neq e$ with $e$ the all-ones vector. 
    This cospectral pair is not generalized cospectral in the sense of \cite{wang2025haemers}.
    We refer to Section \ref{sec: Numerical} for statistics on cospectrality probabilities for random integer matrices of varying dimensionality.    
    } 
\end{example}

\subsubsection*{Outline}
We present the results in Section \ref{sec: Results}.
We give proofs in Sections \ref{sec: ProofsDedekind} and  \ref{sec: ProofsNumber}. 
The algorithm is given in Section \ref{sec: Algorithm}. 
We conclude with numerical experiments in Section~\ref{sec: Numerical}.

\section{Results}\label{sec: Results}

A key fact in our arguments is that the ring of integers in a number field is a Dedekind domain, so we present our results in this setting.
A relevant example to keep in mind is $\bbZ[\sqrt{10}] \subseteq \bbQ(\sqrt{10})$,  which is not a principal ideal domain.  
Other examples where the framework is applicable are the coordinate rings of non-singular curves such as $\bbF_p[x]$ and $\bbR[x,y]/(x^2 + y^2 -1)\bbR[x,y]$.

We introduce the framework in Section \ref{sec: Framework}, present our general results in Section \ref{sec: ResultsDedekind}, and specialize to algebraic number fields in Section \ref{sec: AlgebraicNumber}. 
\subsection{General framework}\label{sec: Framework}
All rings in this paper are commutative and unital. 
A Noetherian integral domain $R$ is a \emph{Dedekind domain} if every nonzero proper ideal $\cI \neq R$ admits a prime factorization, \ie there exist distinct nonzero prime ideals $\cP_1,\ldots,\cP_m \subseteq R$ and integers $e_1,\ldots,e_m \geq 1$ with 
    \begin{equation}
        \cI = \cP_1^{e_1} \cP_2^{e_2} \cdots \cP_m^{e_m}.\label{eq: L_PrimeFactorization} 
        \end{equation}   
    The prime ideals $\cP_i$ and integers $e_i$ in \eqref{eq: L_PrimeFactorization} are unique up to reordering.

    \pagebreak[4]
    Let $K$ be the fraction field of $R$.
    We fix a symmetric matrix $\bX  \in R^{n\times n}$ and assume that there exist $\bY \in R^{n\times n}$ and some nonzero $\bQ\in K^{n\times n}$ with  
    \begin{equation}
        \bX \bQ = \bQ \bY  \quad \textnormal{ and }\quad \bQ^{\T}\bQ \in R^{n\times n}.\label{eq:CalmEel}  
    \end{equation}
    In particular, $\bQ^{\T}\bQ\in R^{n\times n}$ holds if $\bQ$ is orthogonal in the strict sense $\bQ^{\T} \bQ = \bI$, in which case the first condition states that $\bQ^{\T}\bX\bQ = \bY$. 
    Some of our results however do not require the full strength of orthogonality. 
    We will in general hence only assume \eqref{eq:CalmEel} and explicitly state it when we impose stronger assumptions.

    We measure how far $\bQ$ is from being a matrix over $R$ by considering the denominator of its entries. 
    The \emph{level} of $\bQ$ is the ideal $\cL \subseteq R$ defined by 
    \begin{equation}
         \cL  \de \bigl\{r\in R: r\bQ \in R^{n\times n} \bigr\}.\label{eq: Def_L} 
    \end{equation}
    Note that $\cL = R$ if and only if $\bQ \in R^{n\times n}$. 
    
    If $R = \bbZ$, then $\bQ$ has entries in the rational numbers and $\cL \subseteq \bbZ$ is a principal ideal with generator given by the least common denominator of the entries of $\bQ$. 
    In this case, \eqref{eq: Def_L} recovers the definition of the level of a rational matrix used in previous sufficient conditions \cite{wang2016square,qiu2023generalized,wang2006sufficient,van2025Sufficient,wang2026graph}.   
    In such previous works, controlling the prime factorization over $\bbZ$ has been a key aim. 
    The natural analogy for Dedekind domains is the prime factorization of ideals \eqref{eq: L_PrimeFactorization}. 
    An ideal $\cA \subseteq R$ is said to be \emph{divisible} by an ideal $\cB\subseteq R$ if there exists an ideal $\cC \subseteq R$ with $\cB\cC =\cA$.
    In this case, we write $\cB \mid \cA$.
    We write $\cB \nmid \cA$ otherwise.

    Our initial goal will be to control the prime divisors of the level $\cL$. 
    If the level is nontrivial, then we further study the $R$-submodule of $R^n$ defined by 
    \begin{equation}
        \operatorname{Im}_R(\cL \bQ) \de \Bigl\{ \sum_{i=1}^m \ell_i \bQ v_i  : \ell_i \in \cL, \ v_i \in R^n\Bigr\}. \label{eq: Def_ImLQ}
    \end{equation}
    The study of this module will enable constraints on the potential columns of $\bQ$. 

    \subsection{Results over Dedekind domains}\label{sec: ResultsDedekind}
    \subsubsection{Sufficient condition}\label{sec: SufficientCondition}
    Recall that $\phi_\bX(x) = \det(x \bI - \bX)$ is the characteristic polynomial. 
    We define the \emph{discriminant} using the formal derivative $\phi_{\bX}'$ as 
\begin{equation}
    \Delta_{\bX} \de \det\bigl(\phi_{\bX}'(\bX)\bigr).\label{eq:VividRaven}
\end{equation}     
    Note that $\Delta_{\bX} \in R$, so we can consider the ideal $\Delta_{\bX}R\subseteq R$. 
    The following Theorem \ref{thm: MainPlevelP2disc} and Corollary \ref{cor: SufficientSqfree} generalize \cite{wang2016square} from $R =\bbZ$ to Dedekind domains:
    \begin{theorem}\label{thm: MainPlevelP2disc}
        If $\cP \mid \cL$ for some non-zero prime ideal $\cP \subseteq R$ then $\cP^2 \mid \Delta_{\bX}R$.  
    \end{theorem}
   \begin{corollary}\label{cor: SufficientSqfree}
       Suppose that $\Delta_{\bX} \neq 0$ and that in the ideal prime factorization $\Delta_{\bX}R = \prod_{i}\cP_i^{e_i}$ every exponent satisfies $e_i\leq 1$. 
       Then, $\cL = R$. 
   \end{corollary}
    It is conjectured that the discriminant is square-free over $\bbZ$ with non-vanishing frequency for random symmetric integer matrices, such as those arising as the adjacency matrices of random graphs with loops; see Lvov and Van Werde \cite[Conjecture 1.7]{lvov2026satisfaction}.
    Let us note however that $\Delta_{\bX}$ being a square-free integer does not necessarily imply that $\Delta_{\bX}R$ is a square-free ideal when $R$ is the ring of integers of a number field; see Section \ref{sec: AlgebraicNumber}.

    \begin{remark}\label{rem: DiscPoly}
        In general, the \emph{discriminant of a monic polynomial} $f\in R[x]$ can be defined using roots of $f$ in the algebraic closure $\lambda_1,\ldots,\lambda_n\in \overline{K}$ as $\Delta_f \de  \prod_{i<j}  (\lambda_i -\lambda_j)^2$.  
        A direct calculation using Jordan decomposition verifies that we then have $\Delta_{\bX} =  \pm \Delta_{\phi_{\bX}}$. 
        This alternative definition may be useful when comparing to related literature, but we shall see that  \eqref{eq:VividRaven} is the natural definition for our proofs.  
        I thank Nikita Lvov for bringing  \eqref{eq:VividRaven} as well as other formulations to my attention in our recent collaboration \cite[Remark 3.11]{lvov2026satisfaction}.
    \end{remark}
    \begin{remark}\label{rem: SimpleDisc}
        If $\bX$ is the adjacency matrix of a simple graph on $n$ nodes, then it always holds that $4 \mid \Delta_{\bX}$ \cite{wang2016square,wang2024rational}.
        Thus, different from graphs with loops, the discriminant is never square-free for simple graphs. 
        The subsequent results are still applicable to simple graphs, although the algorithm's efficiency will be reduced when the dimensionality is large. 
    \end{remark}
    \subsubsection{Constraints on the matrix $\bQ$}\label{sec: ConstraintsQ}
    We next constrain the possibilities for $\bQ$ and $\cL$ when $\Delta_{\bX}R$ fails to be square-free. 
    We first set up some notation.
      
    Nonzero prime ideals in a Dedekind domain are maximal, so the quotient $\bbF_{\cP} \de R/\cP$ is a field. 
    Consider the factorization of $\phi_{\bX} \bmod \cP$ into powers of irreducible monic polynomials $\gamma_1,\ldots,\gamma_m \in \bbF_{\cP}[x]$: 
    \begin{equation}
        \textstyle\phi_{\bX}  \equiv \prod_{i=1}^m \gamma_{i}^{d_i}  \bmod \cP,  
    \end{equation}
    with $d_1,\ldots,d_m \geq 1$.  
    Define an (approximate) square-root $\Gamma \in \bbF_{\cP}[x]$ by  
    \begin{equation}
        \textstyle \Gamma \de \prod_{i=1}^m \gamma_i^{\lfloor d_i / 2\rfloor}, \label{eq: Def_Gamma} 
    \end{equation}
    where $\lfloor \cdot \rfloor$ is the floor function.
    
    \begin{definition}\label{def: Contractive}
        Fix some $e\geq 1$. 
        A $R$-submodule $\sM \subseteq R^n$ is \emph{$\Gamma$-contractive up to level $\cP^e$} if for every $k\leq e-1$ and $q\in  \sM \cap (\cP^k )^n$ it holds that 
        \begin{equation}
            \Gamma(\bX)q \equiv 0 \bmod \cP^{k+1}. 
        \end{equation}
    \end{definition}
    \begin{definition}\label{def: Iso}
        Fix some $e\geq 1$. 
        A $R$-submodule $\sM \subseteq R^n$ is said to be \emph{isotropic up to level $\cP^e$} if $q_1^{\T} q_2\equiv 0 \bmod \cP^{e}$ for every $q_1,q_2 \in \sM$.
    \end{definition}  
    Isotropy has recently also played a role in sufficient conditions for generalized cospectrality by Guo and Wang \cite{guo2025primary}.
    \begin{definition}\label{def: Krylov}
        The \emph{Krylov space} of a set of vectors $w_1,\ldots,w_m\in R^n$ is given by 
    \begin{equation}
        \sK_{\bX}(w_1,\ldots,w_m) \de  \Bigl\{\sum_{j=1}^m\sum_{i=0}^{n-1} c_{i,j}\bX^i w_j : c_{i,j} \in  R \Bigr\}. \label{eq:TediousElf} 
    \end{equation}
    \end{definition}
    The Krylov space is sometimes also referred to as \emph{control space} \cite{godsil2012controllable,o2016conjecture} and relates to the \emph{walk matrix} considered in some sufficient conditions for generalized spectral characterization \cite{wang2013generalized,van2025Sufficient,wang2026graph} through the image of that matrix. 
    
    It would be equivalent to let the powers $i$ run up to integers greater than $n-1$ in \eqref{eq:TediousElf}. 
    Indeed, it follows from the Cayley--Hamilton theorem that $\bX^n$ can be expressed as a $R$-linear combination of the $\bX^i$ with $i<n$. 
    Thus, the Krylov subspace can also be interpreted as the $R[x]$-module span of the vectors $w_1,\ldots,w_m$ when one equips $R^n$ with the action of $\bX$. 
    The latter perspective has also played a role in recent probabilistic results  \cite{van2025cokernel,lvov2026satisfaction}.

    The contraposition of the following Theorem \ref{thm: LevelBound} can be used to prove upper bounds on the greatest power of $\cP$ that divides the level $\cL$. 
    \begin{theorem}\label{thm: LevelBound}
        Assume that $\cP^e \mid \cL$ for some non-zero prime ideal $\cP \subseteq R$ and $e\geq 1$. 
        Then, there exists $w\in R^n$ with $w\not\equiv 0 \bmod \cP$ for which the Krylov space $\sK_{\bX}(w)$ is $\Gamma$-contractive up to level $\cP^e$ and isotropic up to level $\cP^{2e}$. 
    \end{theorem}    
    It can occur that the bound resulting from contraposition of Theorem \ref{thm: LevelBound} degenerates to infinity if a vector $w$ with the claimed properties exists for every $e\geq 1$. 
    For instance, this is (necessarily) so if $\bX = \bI$. 
    We however prove in Lemma \ref{lem: KrylovIsotropyBound} that the bound is finite when $R$ is the ring of integers of a number field and $\bX\in \bbZ^{n\times n}$ satisfies $\Delta_{\bX} \neq 0$. 
    That is, if all eigenvalues of $\bX$ are distinct.

    We can also restrict the columns of $\bQ$.
    Recall the definition of $\operatorname{Im}_R(\cL \bQ)$ from \eqref{eq: Def_ImLQ}. 
    This module corresponds to a contractive and isotropic Krylov space:
    \begin{theorem}\label{thm: ColumnRestriction}
        Let $e\geq 1$ be an integer such that $\cP^{e}\mid \cL$. 
        Then, there exist $m\geq 1$ vectors $w_1,\ldots,w_m\in R^n$, not all zero modulo $\cP$, whose Krylov space $\sK_{\bX}(w_1,\ldots,w_m)$ is $\Gamma$-contractive up to level $\cP^e$ and isotropic up to level $\cP^{2e}$ and such that 
        \begin{equation}
            \operatorname{Im}_R(\cL \bQ)  = \sK_{\bX}(w_1,\ldots,w_m).\label{eq: Im_Kw1wm}
        \end{equation}
    \end{theorem}

    \subsection{Algebraic number fields}\label{sec: AlgebraicNumber}

    Recall that an \emph{algebraic number field} is a field $K$ of characteristic zero that is finite-dimensional as a $\bbQ$-vector space. 
    Further, recall that the \emph{ring of integers} $\sO_K$ consists of the roots of monic integer polynomials: 
    \begin{equation}
        \sO_K \de \Bigl\{k\in K: \exists m\geq 1,\ \exists c_i\in \bbZ,\  k^m + c_{m-1}k^{m-1} + \cdots + c_0 = 0   \Bigr\}.
    \end{equation}
    The ring of integers of a number field is a Dedekind domain \cite[C1, \S3]{neukirch1999algebraic}. 
    We next apply the results of the preceding Section \ref{sec: ResultsDedekind} with $R=\sO_K$. 
    There, $K$ denoted the fraction field of the Dedekind domain $R$. 
    This is consistent with the current notation since the field $K$ is indeed the fraction field of its ring of integers \cite[C1, \S2]{neukirch1999algebraic}.

    \subsubsection{Sufficient condition}
    Consider an integer symmetric matrix $\bX \in \bbZ^{n\times n}$. 
    Then, the discriminant is also an integer $\Delta_{\bX}\in \bbZ$. 
    If the discriminant is further nonzero, then it admits a unique factorization: 
    \begin{equation}
        \textstyle \Delta_{\bX} = \pm p_1^{h_1}p_2^{h_2}\cdots p_m^{h_m},  \label{eq:IvoryBeetle}  
    \end{equation}
    for integers $h_i \geq 1$ and distinct prime numbers $p_i\in \bbZ_{\geq 2}$. 
    Hence, as ideals,  
    \begin{equation}
         \textstyle \Delta_{\bX}\sO_K = \bigl(p_1\sO_K\bigr)^{h_1}\bigl(p_2\sO_K\bigr)^{h_2} \cdots \bigl(p_m\sO_K\bigr)^{h_m}. \label{eq:GhostlyTin}
    \end{equation}
    It would however be premature to conclude that Corollary \ref{cor: SufficientSqfree} may be used if $h_i \leq 1$ for every $i$. 
    Indeed, the ideals $p_i\sO_K$ are not necessarily prime.
    Moreover, the ideals are not even necessarily square-free: 
    \begin{definition}
        A prime number $p\in \bbZ_{\geq 2}$ is said to be \emph{ramified} in the algebraic number field $K$ if there exists a nonzero prime ideal $\cP \subseteq \sO_K$ with $\cP^2 \mid p\sO_K$. 
    \end{definition}
    Ramified primes always exist if $K\neq \bbQ$, but there are only finitely many \cite[p.207]{neukirch1999algebraic}.
    For nonzero prime $\cP \subseteq \sO_K$, the ideal $\cP \cap \bbZ \subseteq \bbZ$ is again prime and nonzero \cite[C3, p44-45]{marcus1977number}.
    If $p$ is the prime with $\cP\cap \bbZ =  p\bbZ$ then $\cP$ is said to \emph{lie above} $p$.   
    
    \pagebreak[3]
    Consider a symmetric integer matrix $\bX\in \bbZ^{n\times n}$ and suppose that we are given some matrix $\bQ\in K^{n\times n}$ which satisfies \eqref{eq:CalmEel} over $R = \sO_K$. 
    That is, $\bQ^{\T}\bQ \in \sO_K^{n\times n}$ and there exists $\bY \in \sO_{K}^{n\times n}$ with $\bX\bQ = \bQ \bY$.  
    Recall the definition of $\cL$ from \eqref{eq: Def_L}. 
    The following results then follow from Theorem \ref{thm: MainPlevelP2disc} and Corollary \ref{cor: SufficientSqfree}: 
    \begin{corollary}\label{cor: Ramified1}
        If $\cP \mid \cL$ for a non-zero prime ideal $\cP\subseteq \sO_K$, then $\cP$ lies above a prime $p\in \bbZ_{\geq 2}$ with $p\mid \Delta_{\bX}$. 
        Moreover, if $p^2 \nmid \Delta_{\bX}$ then the prime $p$ is ramified in $K$ and it holds that $\cP^2 \mid p\sO_K$.  
    \end{corollary}
    \begin{corollary}\label{cor: Ramified2}
        Assume that $\Delta_{\bX}\neq 0$ and that none of the primes in the factorization \eqref{eq:IvoryBeetle} is ramified in $K$. 
        Further, assume that all $h_i$ in \eqref{eq:IvoryBeetle} satisfy $h_i\leq 1$. 
        Then, $\cL = \sO_K$.
    \end{corollary}

    The foregoing results all concerned the general setting of \eqref{eq:CalmEel}. 
    Our main motivation however comes from matrices that are orthogonal in the strict sense that $\bQ^{\T}\bQ = \bI$. 
    Under an additional assumption on the number field, this allows us to deduce more restrictive constraints on the matrix $\bQ$. 

    The number field $K$ is said to be \emph{totally real} if every field embedding $\sigma: K \to \bbC$ satisfies $\sigma(K) \subseteq \bbR$.
    In particular, this is the case for the splitting field of a polynomial $f\in \bbQ[x]$ if all roots of $f$ are real. 
    \begin{proposition}\label{prop: TotallyRealSigned}
        Consider a totally real number field $K$. 
        Then, for any fixed ideal $\cL\subseteq \sO_K$, there are only finitely many matrices $\bQ \in K^{n\times n}$ with $\bQ^{\T} \bQ = \bI$ and level ideal $\cL$. 
        Moreover, the only such matrices with $\cL = \sO_K$ are signed permutations. 
    \end{proposition}

    Combining Corollary \ref{cor: Ramified2} and Proposition \ref{prop: TotallyRealSigned} over some totally real algebraic number field $K$ gives a sufficient condition to rule out nontrivial cospectral mates over $K$ in the sense of Definition \ref{def: CospectralOverK}. 

    \begin{example}\label{example: Q_sqrt_d}
        Consider a positive square-free integer $d\geq 2$. 
        Then $K = \bbQ(\sqrt{d})$ is totally real. 
        The ramified primes are the prime divisors of $d$, and also $p=2$ if $d\not\equiv 1 \bmod 4$ (see, \eg, \cite[C1, p.15 
        ]{neukirch1999algebraic}, \cite[C3, p.199, Theorem 2.6]{neukirch1999algebraic}): 
        \begin{equation}
            \sP_d  = 
            \begin{cases}
             \{\textnormal{prime }p : p\mid d \} & \textnormal{ if }d\equiv 1 \bmod 4, 
             \\
             \{2 \}\cup \{\textnormal{prime }p : p\mid d \} & \textnormal{ if }d\not\equiv 1 \bmod 4.
            \end{cases}
        \end{equation} 
    Corollary \ref{cor: Ramified2} and Proposition \ref{prop: TotallyRealSigned} now yield a sufficient condition applicable to symmetric matrices $\bX \in \bbZ^{n\times n}$ with square-free discriminant $\Delta_{\bX}$ whose prime divisors $p \mid \Delta_{\bX}$ satisfy $p\not\in \sP_d$.
        When applicable, we may conclude that the only orthogonal matrices $\bQ\in K^{n\times n}$ with $\bQ^{\T}\bX \bQ \in \sO_{K}^{n\times n}$ are signed permutations.
    \end{example}
     
    One notable example of a totally real number field is the field found by adjoining a root of the characteristic polynomial $\phi_{\bX}$.   
    Let us note however that the assumptions of Corollary \ref{cor: Ramified2} are essentially never applicable in that setting. 
    In particular, if $\Delta_{\bX}$ is square-free, then every prime divisor $p\mid \Delta_{\bX}$ ramifies in the latter field if $\phi_{\bX}$ is irreducible \cite[p.79, Theorem 34]{marcus1977number} (see also the proof of Lemma \ref{lem: IrreducibleDegree1}), and the assumption becomes $\Delta_{\bX} = \pm 1$. 
    Matrices with discriminant $\pm 1$ are exceptional:

    \begin{proposition}\label{prop: Discpm1}
        A symmetric matrix $\bX\in \bbZ^{n\times n}$ with $n>1$ satisfies $\lvert \Delta_{\bX} \rvert =  1$ if and only if $n=2$ and it is a diagonal matrix of the following form for some $a\in \bbZ$: 
        \begin{equation}
            \bX = \begin{pmatrix}
                a & 0 \\ 
                0 & a \pm 1
            \end{pmatrix}.
        \end{equation}
    \end{proposition}
    A proof is given in Appendix \ref{sec: MatrixDiscpm1}. 
    That the sufficient condition fails when the field is chosen depending on the matrix is not surprising. 
    Indeed, if one considers the field $K$ that arises from the splitting field after the tower of quadratic extensions necessary to normalize the eigenvectors into the unit vectors, then $\bX$ can be related through an orthogonal matrix with entries in $K$ to the diagonal matrix of eigenvalues, which has entries in $\sO_K$.

    \subsubsection{Algorithmic search for cospectral mates}\label{sec: AlgorithmicSearch}
    If  Corollary \ref{cor: Ramified2} is not applicable in some given setting, then it may be possible to use our other results to constrain the potential orthogonal matrices to a finite set. 
    This can be used to find all cospectral mates over $K$ of a given matrix $\bX \in \bbZ^{n\times n}$ with $\Delta_{\bX} \neq 0$. 
    We here describe when the resulting algorithm is reasonably efficient.
    
    Let us note from the start that it will be important that the field $K$ is not too large. 
    This is because we have to solve problems in computational algebraic number theory that are otherwise prohibitively expensive.  
    
    \begin{assumption}\label{ass: SmallNumberField}
        The number field $K$ is totally real and is such that we can practically run algorithms over $\sO_K$. 
        (\eg computing the list of ramified primes, determining the prime ideals $\cP \subseteq \sO_K$ with $\cP^2 \mid p\sO_K$ for some given $p\in \mathbb{Z}$,...)  
    \end{assumption}
    
    Even if the number field is small, like $K=\bbQ$, we will have to determine all primes $p$ that satisfy $p^2 \mid \Delta_{\bX}$. 
    To our knowledge, there is no way to do this that is significantly more efficient than factorizing the discriminant. 
    Prime factorization is tractable for some moderate-sized matrices, such as those in Example \ref{ex: CospectralInteger}, but prohibitively expensive for large matrices.
    
    Probabilistic heuristics however suggest the relevant primes should typically all be small. 
    Indeed, modeling $\Delta_{\bX}$ as a random integer gives a probability of order $1/p^2$ for divisibility by $p^2$, which is summable. 
    (See also \cite[Conjecture 1.7]{lvov2026satisfaction} for refined predictions.)
    We assume knowledge of an a-priori bound:
    \begin{assumption}\label{ass: SmallPrimes}
        We are given an integer $M\geq 1$ such that every prime number $p$ with $p^2\mid \Delta_{\bX}$ satisfies that $p\leq M$.   
    \end{assumption}
    In practice, if a rigorous upper bound is unknown, then one can simply substitute a large value of $M$ for a heuristic algorithm.  
    We can then not guarantee the completeness of the set of cospectral mates, only that we find all orthogonal matrices whose level only contains ideals lying above primes $p\leq M$.

    The efficiency of the algorithm also depends on the approximate square root $\Gamma$ defined in \eqref{eq: Def_Gamma}:  
    \begin{assumption}\label{ass: DegGamma}
        It holds for every prime ideal $\cP \subseteq \sO_K$ with $\cP^2 \mid \Delta_{\bX}\sO_K$ that the polynomial $\Gamma\in \bbF_{\cP}[x]$ has low degree.    
    \end{assumption}
    Different from Assumption \ref{ass: SmallPrimes}, the degree of $\Gamma$ does not have to be known from the start.  
  
    \pagebreak[3]
    Subject to the foregoing assumptions, Algorithm \ref{alg:orth_matrices} in Section \ref{sec: Algorithm} enables finding all cospectral mates of $\bX$ over $K$. 
    This is done essentially by bounding the level and next searching through the finitely many orthogonal matrices that remain possible, although implementing this idea requires developing additional theory. 
    At any rate, we have the following guarantee: 
      
    \begin{theorem}\label{thm: Algorithm}
        Consider an integer symmetric matrix $\bX\in \bbZ^{n\times n}$ with $\Delta_{\bX} \neq 0$ and a totally real number field $K$.
        Then,  Algorithm \ref{alg:orth_matrices} terminates in finite time and outputs all matrices $\bQ \in K^{n\times n}$ with $\bQ^{\T}\bQ = \bI$ and $\bQ^{\T} \bX \bQ \in \sO_K^{n\times n}$.   
    \end{theorem}

    This guarantee does not require the foregoing assumptions, but they are nonetheless crucial for the interpretation. 
    In particular, the subroutines involved are reasonably tractable only if the assumptions are applicable.

    For comparison, it could also be shown by more elementary means that there are only finitely many symmetric matrices $\bY \in \bbZ^{n\times n}$ that are cospectral to $\bX$. 
    Indeed, since the trace of a matrix is the sum of the eigenvalues,
    \begin{equation}
        \sum_{i,j=1}^n\bX_{i,j}^2 = \operatorname{Tr}(\bX^2) = \operatorname{Tr}(\bY^2) = \sum_{i,j=1}^n\bY_{i,j}^2.\label{eq:RoyalSun}
    \end{equation}
    In principle, this also enables finding all such matrices $\bY$ by brute force in finite time.
    The usability of such an exhaustive method is however limited due to the number of candidate matrices exploding as a function of $\sum_{i,j}\bX_{i,j}^2$.

  Even if one restricts $\bX$ and $\bY$ to have entries in $\{0,1 \}$, the number of candidate matrices to compare with in \eqref{eq:RoyalSun} grows exponentially in $n^2$.     
  It took a non-trivial effort when Brouwer and Spence determined the fraction of cospectral graphs on $n\leq 12$ nodes through exhaustive enumeration \cite{brouwer2009cospectral}, extending earlier work by Godsil and McKay \cite{godsil1976some} for $n\leq 9$ and by Haemers and Spence for $n \leq 11$ \cite{haemers2004enumeration}. 
  
  Algorithm \ref{alg:orth_matrices} can be applied to to integer matrices that simultaneously have large entries and high dimensionality. 
  In fact, due to phenomena related to Remark \ref{rem: SimpleDisc}, such applications with general symmetric integer matrices will typically run somewhat faster than applications to simple graphs. 
  Numerical data is given in Section \ref{sec: Numerical}. 
    \begin{remark}
        One can not guarantee that there are only finitely many orthogonal matrices $\bQ \in K^{n\times n}$ with $\bQ^{\T}\bX \bQ \in \sO_K^{n\times n}$ if $\Delta_{\bX} = 0$.
        For example, consider $\bX = \bI$. 
        In fact, infinitely many orthogonal matrices with  $\bQ^{\T}\bX \bQ \in \sO_K^{n\times n}$ appears to be the typical behavior when $\Delta_{\bX} = 0$, at least when $K$ is sufficiently large, since one can act by rotation on an eigenspace with repeated eigenvalue.
        Thus, the assumption in Theorem \ref{thm: Algorithm} is necessary to guarantee a finite output.

        On the other hand, these examples arise from the fact that there are infinitely many orthogonal matrices with $\bX = \bQ^{\T}\bX \bQ$ and do not actually give rise to an infinitude of cospectral mates. 
        (Indeed, there can only be finitely many cospectral mates due to \eqref{eq:RoyalSun}.) 
        Theory and algorithms for matrices with repeated eigenvalues remain an interesting question for future work.  
    \end{remark}
    \section{Matrices over a Dedekind domain --- Proofs for \texorpdfstring{Section \ref{sec: ResultsDedekind}}{Section}}\label{sec: ProofsDedekind} 
    Throughout this section, we adopt the general framework from Section \ref{sec: Framework}. 
    Thus, let $\bX,\bY\in R^{n\times n}$ be matrices over a Dedekind domain with $\bX$ symmetric, and  consider a matrix $\bQ \in K^{n\times n}$ over the fraction field such that \eqref{eq:CalmEel} is satisfied. 

    We set up some general preliminaries in Section \ref{sec: Preliminaries} that are repeatedly used in the subsequent arguments. 
    The arguments for Theorems \ref{thm: MainPlevelP2disc}, \ref{thm: LevelBound}, and \ref{thm: ColumnRestriction}  are subsequently given in Sections \ref{sec: Proof_DivLevelDivDelta} to \ref{sec: ProofMainP2disc}.   
    \subsection{Preliminaries}\label{sec: Preliminaries}
    \subsubsection{Structure of the image module}\label{sec: StructureImage}
    Recall the definitions of $\cL$ and $\operatorname{Im}_R(\cL \bQ)$ from \eqref{eq: Def_L} and \eqref{eq: Def_ImLQ}, respectively.   
    The following two results show how the assumption \eqref{eq:CalmEel} is reflected in the structure of $\operatorname{Im}_R(\cL \bQ)$.  
    \begin{lemma}\label{lem: Restrict}
        The matrix $\bX$ restricts to a $R$-linear map on $\Im_{R}(\cL\bQ)$,   
        \begin{equation}
            \bX\mid_{\Im_{R}(\cL \bQ)}:\Im_{R}(\cL \bQ)  \to \Im_{R}(\cL\bQ): v \mapsto \bX v.\label{eq:GoodKing}  
        \end{equation}   
    \end{lemma}
    \begin{proof}
        By \eqref{eq:CalmEel}, there exists $\bY \in R^{n\times n}$ with $\bX (\ell\bQ) = (\ell\bQ)\bY$ for every $\ell \in \cL $.  
        Recall from the definition \eqref{eq: Def_ImLQ} that every vector in $\operatorname{Im}_R(\cL \bQ)$ is a $R$-linear combination of vectors of the form $(\ell \bQ) v$ with $v\in R^n$ and $\ell \in \cL$. 
        It hence follows that  $\bX(\Im_{R}(\cL \bQ)) \subseteq \Im_{R}(\cL \bQ)$, as desired.  
    \end{proof}
    \begin{lemma}\label{lem: IsotropyImQ}
        It holds for every $q_1,q_2 \in \operatorname{Im}_R(\cL \bQ)$ that $q_1^{\T} q_2 \equiv 0 \bmod \cL^2$.
    \end{lemma}
    \begin{proof}
        By \eqref{eq:CalmEel}, we have $v_1^{\T} \bQ^{\T} \bQ v_2 \in R$ for all $v_1,v_2 \in R^n$.
        Hence, $(\ell_1v_1)^{\T} \bQ^{\T} \bQ (\ell_2 v_2) \in \cL^2$ for every $\ell_1,\ell_2 \in \cL$.  
        The claim is hence immediate from the definition \eqref{eq: Def_ImLQ}. 
    \end{proof}
    
    \subsubsection{Ideal arithmetic}
    
    We will rely on the ideal arithmetic of Dedekind domains and hence recall the properties for easy reference. 
    Recall that a \emph{fractional ideal} is a $R$-submodule $\cF \subseteq K$  of the fraction field  $K$ such that  $r\cF \subseteq R$ for some non-zero $r\in R$. 
    Recall that $\cF^{-1} \de \{k\in K : k\cF \subseteq R  \}$.  
    \begin{proposition}\label{prop: IdealArithmetic}
        Every nonzero fractional ideal $\cF \subseteq K$ admits a factorization $\cF = \cP_{1}^{f_1} \cdots \cP_{m}^{f_m}$ with $f_j  \in \bbZ$ and distinct prime ideals $\cP_j \subseteq R$, which is unique up to reordering of the factors. 
        If $\cG = \cP_{1}^{g_1} \cdots \cP_{m}^{g_m}$ is another fractional ideal, then   
        \begin{equation}
            \textstyle \cF + \cG =  \prod_{j=1}^m \cP_{j}^{\min\{f_j, g_j \}}, \ \ \cF \cap \cG = \prod_{j=1}^m \cP_{j}^{\max\{f_j,g_j\}}, \ \ \cF  \cG = \prod_{j=1}^m \cP_j^{f_j + g_j}
        \end{equation}
        In particular, $\cF \subseteq \cG$ if and only if $f_j \geq g_j$ for every $j$.  
    \end{proposition}
    \begin{proof}
        See \cite[Chapter 5, Theorem 11]{zariski2013commutative}, \cite[Chapter 3]{marcus1977number}, or \cite[Chapter 16]{dummit_foote_2004}. 
    \end{proof}

    \begin{lemma}\label{lem: PrimeFactorization}
        Suppose that $\cL \neq R$ and consider a nonzero prime divisor $\cP \mid \cL$.  
        Then, there exists at least one $\ell \in \cL$ with 
        $
            \ell\bQ \not\equiv 0 \bmod  \cP.
        $        
    \end{lemma}
    \begin{proof}
        Consider the prime factorization $\cL = \prod_{j=1}^m \cP_j^{e_j}$ and suppose to the contrary that
        $
            \cL \times  (\sum_{v,w=1}^n \bQ_{v,w}R) \subseteq \cP_iR
        $ for some $\cP_i$ with $e_i \geq 1$.
        Then, by the inclusion and multiplication properties of fractional ideals,
        \begin{equation}
            \textstyle (\cP_i^{e_i -1} \prod_{j\neq i} \cP_j^{e_j}) \times  (\sum_{v,w=1}^n \bQ_{v,w}R) \subseteq R.
        \end{equation}
        This would however contradict that $\cL$ is the maximal (non-fractional) ideal satisfying  $\cL \times (\sum_{v,w=1}^n \bQ_{v,w}R ) \subseteq R$; recall the definition \eqref{eq: Def_L}. 
    \end{proof}

    \subsection{Discriminant and characteristic polynomial}\label{sec: Proof_DivLevelDivDelta}
     Among other things, we here prove a first step towards Theorem \ref{thm: MainPlevelP2disc} by establishing that $\cP \mid \cL$ implies that $\cP \mid \Delta_{\bX}R$.
     (The strengthening to $\cP^2 \mid \Delta_{\bX}R$ is proved in Section \ref{sec: ProofMainP2disc}.)  
     We rely on the following classical property of the discriminant:
    \begin{lemma}\label{lem: DiscriminantField}
        Let $\bbF$ be a field. 
        Then, $g^2\mid \phi_{\bM}$ for $g\in \bbF[x]$ with $\operatorname{deg}(g)\geq 1$ implies that $\Delta_{\bM} = 0$ in $\bbF$. 
        Moreover, the converse also holds if $\bbF$ is a perfect field.   
    \end{lemma}
    \begin{proof}
        Consider the factorization of the characteristic polynomial, \ie $\phi_{\bM} = \prod_{j}g_j^{d_j}$ with $g_j \in \bbF[x]$ coprime, monic, and irreducible. 
        Then, 
        \begin{equation}
            \textstyle \phi_{\bM}' = \sum_j d_j g_j^{d_j -1} g_j'\prod_{i\neq j} g_i^{d_i}.\label{eq:CalmParrot}
        \end{equation}
        Note that $f\in \bbF[x]$ satisfies $\det(f(\bM)) = 0$ if and only if $f$ has shared roots with $\phi_{\bM}$ in the latter's splitting field. 
        Shared roots are implied by shared irreducible divisor and this is an equivalence if $\bbF$ is perfect.
        The result hence follows from the definition $\Delta_{\bM} = \det(\phi_{\bM}'(\bM))$ since \eqref{eq:CalmParrot} yields that $\phi_{\bM}'\equiv 0 \bmod g_j$ if and only if $d_j\geq 2$. 
    \end{proof}

    Let us assume that $\cL\neq R$. 
    Motivated by Lemma \ref{lem: DiscriminantField}, our goal in the subsequent results is to study the characteristic polynomial of $\bX$ over $\bbF_\cP = R/\cP$ for a prime ideal $\cP \mid \cL$.  
    We rely on the properties of $\operatorname{Im}_R(\cL \bQ)$ from Section \ref{sec: StructureImage}.   

    Let $\sQ_0 \subseteq \bbF_{\cP}^n$ be the subspace found as the reduction of $\operatorname{Im}_R(\cL \bQ)$ modulo $\cP$.
    Further, for every $k \geq 1$ let $\sQ_k \subseteq \bbF_{\cP}^n$ be the reduction modulo $\cP$ of $\{q\in R^n: \forall r\in \cP^k,\,  rq \in \operatorname{Im}_R(\cL\bQ)   + (\cP^{k+1})^n  \}$.
    This defines a nested sequence of subspaces: 
    \begin{equation}
        \sQ_0 \subseteq \sQ_1 \subseteq \sQ_2 \subseteq \cdots \label{eq:ColdBird} 
    \end{equation}
    Informally, the vectors in the difference $\sQ_{k}\setminus \sQ_{k-1}$ can be interpreted as the new directions that are observed when going from the reduction of $\operatorname{Im}_R(\cL \bQ)$ modulo $\cP^k$ to the reduction modulo $\cP^{k+1}$. 
        
    \begin{lemma}\label{lem: NonEmpty}
        Assume that $\cP \mid \cL$. 
        Then, $\sQ_k \neq 0$ for all $k\geq 0$.   
    \end{lemma}
    \begin{proof}
        Due to the inclusions \eqref{eq:ColdBird}, it suffices to consider $k=0$. 
        The latter case follows from Lemma \ref{lem: PrimeFactorization} since $\ell \bQ\not\equiv 0 \bmod \cP$ for some $\ell \in \cL$ implies that $\operatorname{Im}_R(\cL \bQ)$ contains at least one vector with nonzero reduction modulo $\cP$. 
    \end{proof}
    Given a subspace $V\subseteq \bbF_{\cP}^n$, the \emph{orthogonal dual} is the subspace defined by $V^{\perp} \de \{w\in \bbF_{\cP}^n : \forall v\in V,\, w^\T v = 0 \}$. 
    A subspace with $V \subseteq V^{\perp}$ is called \emph{isotropic}.
    Such subspaces are precisely the reductions modulo $\cP$ of submodules of $R^n$ that are isotropic modulo $\cP$ in the sense of Definition \ref{def: Iso}. 
    \begin{lemma}\label{lem: IsotropyQk}
        Assume that $\cP^e \mid \cL$. 
        Then, $\sQ_k$ is isotropic for every $k \leq e-1$. 
    \end{lemma}
    \begin{proof}
        Let $q_1,q_2 \in \{q\in R^n: \forall r\in \cP^k,\,  rq \in \operatorname{Im}_R(\cL\bQ)  + (\cP^{k+1})^n  \}$. 
        Then, by definition of $\sQ_k$, it suffices to show that $q_1^{\T}q_2 \equiv 0 \bmod \cP$. 
        
        Using Lemma \ref{lem: IsotropyImQ}, it holds for all $r_1,r_2 \in \cP^k$ that the vectors $v_i \de r_iq_i$ satisfy $v_1^{\T}v_2 \equiv 0 \bmod \cP^{2k+1}$. 
        Thus, we have an ideal inclusion $\cP^{2k} (q_1^{\T} q_2 R) \subseteq \cP^{2k+1}$. 
        The ideal arithmetic from Proposition \ref{prop: IdealArithmetic} then yields that $q_1^{\T} q_2 R \subseteq \cP$, as desired.
    \end{proof}
    Given a polynomial $g\in \bbF_{\cP}[x]$, let $\sV_g  \subseteq\bbF_{\cP}^n$ denote the primary subspace: 
    \begin{equation}
         \sV_g \de \bigl\{v\in \bbF_\cP^n: \exists j\geq 1,\,  g(\bX)^j v  = 0 \bigr\}.
    \end{equation}
    Let it be understood that $\operatorname{dim}_\cP(\cdot)$,  $\operatorname{Im}_{\cP}(\cdot)$, and $\operatorname{ker}_{\cP}(\cdot)$ refer to the respective operations over the field $\bbF_{\cP}$.    
    So, for instance, $\ker_{\cP}(\bM) \subseteq\bbF_{\cP}^n$ for a matrix $\bM\in R^{n\times n}$ refers to the kernel over $\bbF_{\cP}$ of the matrix $\bM \bmod \cP$. 
    
    \begin{lemma}\label{lem: Psi}
        Assume that $\cP^e \mid \cL$.  
        Then, for every $k\leq e-1$ and $g\in \bbF_{\cP}[x]$,  
        \begin{equation}
            \dim_{\cP}(\sV_{g}) \geq 2\operatorname{dim}_{\cP}(\sQ_k\cap \sV_{g}).\label{eq:GladClaw} 
        \end{equation}
    \end{lemma}
    \begin{proof}         
        Note that $v\not\in \sV_{g}$ implies that $g(\bX)^jv \not\in \sV_{g}$ for every $j$.
       In other words, a vector $v\in \bbF_{\cP}^n$ can only satisfy $g(\bX)^jv \in \sV_g$ if $v\in \sV_g$. 
       Fix some sufficiently large $j\geq 1$ with $\operatorname{ker}_{\cP}(g(\bX)^j) = \sV_g$ and note that we then have $\Im_{\cP}(g(\bX)^j)\cap  \sV_g = 0$.
       In particular,  
        \begin{equation}
            \Im_{\cP}(g(\bX)^j)\cap  (\sQ_{k} \cap \sV_g) = 0.\label{eq:FairLobster}
        \end{equation} 
        
    Recall from Lemma \ref{lem: IsotropyQk} that $\sQ_k  \subseteq \sQ_k^{\perp}$.
    In particular, $\sQ_k\cap \sV_{g} \subseteq (\sQ_k \cap \sV_{g})^{\perp}$.
    Further, using that $\bX$ is symmetric and that $g(\bX)^j$ annihilates every vector in $\sV_g$ by definition of $j$, we have $q^{\T} g(\bX)^jv =0 $ for every $q\in \sV_{g}$ and $v\in \bbF_{\cP}^n$. 
    Thus,  $\Im_{\cP}(g(\bX)^j) \subseteq  \sV_{g}^{\perp}.$
    In particular, we also have that $\Im_{\cP}(g(\bX)^j) \subseteq (\sQ_k\cap \sV_{g})^{\perp}$.

         We have proved that $\sQ_k\cap \sV_{g}$ and $\Im_{\cP}(g(\bX)^j)$ are both subspaces of $(\sQ_k\cap \sV_{g})^{\perp}$, and these subspaces are independent by \eqref{eq:FairLobster}.
        Hence, 
        \begin{equation}
            \dim_{\cP}\bigl( \sQ_k\cap \sV_{g} \bigr) + \dim_{\cP}\bigl(\Im_{\cP}(g(\bX)^j) \bigr) \leq \dim_{\cP} \bigl( (\sQ_k\cap \sV_{g})^{\perp} \bigr). 
        \end{equation}
        Note that $\dim_{\cP}(\Im_{\cP}(g(\bX)^j)) = n - \dim_{\cP}(\ker_\cP(g(\bX)^j)) = n - \dim_{\cP}(\sV_g)$ and that $\dim_{\cP}((\sQ_k\cap \sV_{g})^{\perp}) = n - \dim_{\cP}(\sQ_k\cap \sV_{g})$ to find \eqref{eq:GladClaw}.  
    \end{proof}
    
    The well-definedness of the restriction of $\bX$ to $\operatorname{Im}_R(\cL \bQ)$ from Lemma \ref{lem: Restrict} implies that $\bX \bmod \cP$ restricts to $\sQ_k$ for every $k$. 
    Let $\psi_k(x) \in \bbF_{\cP}[x]$ denote the characteristic polynomial of this restriction $\bX \mid_{\sQ_k}$. 
    Note that $\operatorname{deg}(\psi_k)\geq 1$ by Lemma \ref{lem: NonEmpty}. 

    \begin{lemma}\label{lem: Sq}
        Assume that $\cP^e \mid \cL$ and consider some $k\leq e-1$. 
        Then, $\psi_k^2 \mid \phi_{\bX}$ as polynomials over $\bbF_{\cP}$. 
    \end{lemma}
    \begin{proof}
        Consider some irreducible factor $g\in \bbF_{\cP}[x]$ of $\psi_{k}$ and let $h\geq 1$ be maximal with $g^h \mid \psi_k$. 
        It then suffices to show that $g^{2h} \mid \phi_{\bX}$. 

        By the primary decomposition theorem \cite[C6, Thm.12; C7, Thm.4]{hoffmann1971linear}, the power of $g$ in the characteristic polynomial is proportional to the dimension of the primary subspace: 
        \begin{equation}
            \sup\{d\geq 0: g^d \mid \phi_{\bX} \} = \dim_{\cP}(\sV_g)/ \deg(g).\label{eq:GiddyCity} 
        \end{equation}
        The exact same argument applies with $\phi_{\bX}$ replaced by $\psi_k$, as the latter is the characteristic polynomial of the restriction of $\bX$ to $\sQ_k$: 
        \begin{equation}
            h = \dim_{\cP}(\sQ_k\cap \sV_{g})/\deg(g). \label{eq:RustyUrsa} 
        \end{equation}
        It now follows from Lemma \ref{lem: Psi} that $g^{2h}\mid \phi_{\bX}$, as desired.  
    \end{proof}
    \begin{corollary}
        Assume that $\cP \mid \cL$.
        Then, $\cP \mid \Delta_{\bX}R$. 
    \end{corollary}
    \begin{proof}
        Combine Lemmas \ref{lem: DiscriminantField} and \ref{lem: Sq} and note that $\cP \mid \Delta_{\bX}R$ if and only if $\Delta_{\bX}$ vanishes in $\bbF_{\cP}$.   
    \end{proof}
    \subsection{Krylov spaces --- Proof of Theorems \ref{thm: LevelBound} and \ref{thm: ColumnRestriction}} 
    Note that it holds for every $w_1,\ldots,w_m\in R^n$ and $i\leq m$ that $\sK_{\bX}(w_i) \subseteq \sK_{\bX}(w_1,\ldots,w_m)$.
    Hence, considering that $\Gamma$-contractivity and isotropy are inherited by subspaces, it suffices to prove Theorem \ref{thm: ColumnRestriction} since the latter implies Theorem \ref{thm: LevelBound}.   
    \begin{proof}[Proof of \texorpdfstring{Theorem \ref{thm: ColumnRestriction}}{Theorem}]
        By definition, Dedekind domains are Noetherian.
        Hence, every submodule of $R^n$ is finitely generated.  
        Let $w_1,\ldots,w_m \in \operatorname{Im}_R(\cL \bQ)$ be a set of generators for $\operatorname{Im}_R(\cL \bQ) \subseteq R^n$. 

        That the $w_i$ generate $\operatorname{Im}_R(\cL \bQ)$ implies that $\operatorname{Im}_R(\cL \bQ) \subseteq \sK_{\bX}(w_1,\ldots,w_m)$. 
        On the other hand, recall from Lemma \ref{lem: Restrict} that the action of $\bX$ restricts to $\operatorname{Im}_R(\cL \bQ)$. 
        Thus, recalling  Definition \ref{def: Krylov}, it follows from $w_1,\ldots,w_m \in \operatorname{Im}_R(\cL\bQ)$ that $\sK_{\bX}(w_1,\ldots,w_m) \subseteq \operatorname{Im}_R(\cL \bQ)$. 
        This proves the equality \eqref{eq: Im_Kw1wm}. 

        It remains to prove that $\cP^e \mid \cL$ implies that $\operatorname{Im}_R(\cL \bQ)$ is $\Gamma$-contractive up to level $\cP^e$ and isotropic up to level $\cP^{2e}$. 
        Isotropy is immediate from Lemma \ref{lem: IsotropyImQ}, so only contractivity remains. 
        Pick some arbitrary $v \in \operatorname{Im}_R(\cL \bQ)$ with $v \equiv 0 \bmod \cP^k$ for some $k\leq e-1$.
        Then, recalling the definition of contractivity from Definition \ref{def: Contractive}, it has to be shown that $\Gamma(\bX) v \equiv 0 \bmod \cP^{k+1}$.

        Recall that any nontrivial quotient of a Dedekind domain is a principal ideal ring \cite[p.278]{zariski2013commutative}. 
        Let $r \in \cP^k$ be an element that reduces to a generator of the ideal $\cP^k/\cP^{k+1}$ in $R/\cP^{k+1}$.
        Then, every vector in $(\cP^k)^n$ is congruent to a multiple of $r$ modulo $\cP^{k+1}$. 
        In particular, $v \equiv r w \bmod \cP^{k+1}  $ for this $r\in \cP^k$ and $w\in R^n$. 
        
        Note that $w \bmod \cP$ is an element of $\sQ_k$, by definition of $\sQ_k$.  
        Here, by Lemma \ref{lem: Sq}, the characteristic polynomial $\psi_k$ of $\bX$ on $\sQ_k$ satisfies $\psi_k^2 \mid \phi_{\bX}$ over $\bbF_{\cP}$.
        By the definition \Cref{eq: Def_Gamma} of $\Gamma$ we hence have $\psi_k\mid \Gamma$. 
        Using that $\psi_k(\bX) w \equiv 0 \bmod \cP$ by the Cayley--Hamilton theorem, it follows that $\Gamma(\bX) w \equiv 0 \bmod \cP$. 
        Thus, using that $v \equiv rw \bmod \cP^{k+1}$ with $r\in \cP^k$, we have $\Gamma (\bX) v \equiv 0 \bmod \cP^{k+1}$, as desired.         
    \end{proof}

    \subsection{Divisibility by \texorpdfstring{$\cP^2$}{P2} --- Proof of \texorpdfstring{Theorem \ref{thm: MainPlevelP2disc}}{Theorem}}\label{sec: ProofMainP2disc}
    
    We exploit the representation $\Delta_{\bX} = \det(\phi_{\bX}'(\bX))$ and that the determinant of a matrix vanishes modulo $\cP^k$ if and only if the matrix is sufficiently degenerate, in the following sense: 
    \begin{lemma}\label{lem: Smith}
        Consider a matrix $\bM \in R^{n\times n}$ and an integer $k\geq 1$. 
        Then, $\det(\bM)\equiv 0 \bmod \cP^{k}$ if and only if there exist vectors $v_1,\ldots,v_i \in R^{n}$ whose reductions modulo $\cP$ are linearly independent over $\bbF_{\cP}$ and integers $\lambda_1,\ldots,\lambda_i \geq 1$ such that
        \begin{equation}
            \textstyle\sum_{j=1}^i \lambda_j \geq k \quad  \textnormal{and}\quad  \bM v_j \equiv 0 \bmod \cP^{\lambda_j}\textnormal{ for all }j\leq i.\label{eq:KindMusic}
        \end{equation}  
    \end{lemma}
    \begin{proof}
        Every nontrivial quotient of a Dedekind domain is a principal ideal ring \cite[p.278]{zariski2013commutative}. 
        In particular, matrices over $R/\cP^{k}$ hence admit a Smith normal form \cite[Theorem 15.24]{brown1993matrices}. 
        That is, there exist $d_1,\ldots,d_n \in R$ and invertible $\bV,\bU \in \operatorname{GL}_n(R/\cP^k)$  such that $\bM \bmod \cP^k$ can be reduced to diagonal form: 
        \begin{equation}
            \bM \equiv \bU \bD \bV \bmod \cP^{k} \quad \textnormal{ with }\quad \bD = \operatorname{diag}(d_1,\ldots,d_n). \label{eq:HappySet} 
        \end{equation}
        Multiplication by $\operatorname{GL}_n(R/\cP^k)$ preserves both the event where $\det(\bM) \equiv 0 \bmod \cP^k$ and the event where there exists a system satisfying \eqref{eq:KindMusic}. 
        Consequently, it suffices to prove Lemma \ref{lem: Smith} when $\bM$ is replaced by the diagonal matrix $\bD$.  

        The case with diagonal matrices is elementary. 
        Indeed, for every $j$, let $\lambda_j\geq 0$ be the power of $\cP$ in the factorization of the ideal $d_j R$. 
        Then, the arithmetic for inclusion and product of ideals in Proposition \ref{prop: IdealArithmetic} yields that $\det(\bD) \equiv 0 \bmod \cP^k$ if and only if $\sum_{j=1}^n \lambda_j\geq k$. 
        Let $e_1,\ldots,e_n \in R^n$ be the standard basis vectors. 
        Taking $v_j = e_{r_j}$ with $r_j$ the $j$th index with $\lambda_{r_j}\geq 1$ then proves that \eqref{eq:KindMusic} follows from $\det(\bD) \equiv 0 \bmod \cP^k$. 
        The other direction follows similarly. 
    \end{proof}

   \begin{lemma}\label{lem: LiftSolutionSquared}
        Assume that $\cP^e \mid \cL$. 
        Consider some polynomial $f\in R[x]$, not necessarily monic, such that 
        \begin{equation}
            \ker_{\cP}\bigl(f(\bX)\bigr) \subseteq   \sQ_e. \label{eq:OilyHotel}
        \end{equation}
        Then, it holds for every $v\in \sQ_{e-1}$ with $f(\bX) v \equiv 0 \bmod \cP$ that there exists $w \in R^n$ with $w \equiv v \bmod \cP$ and $f(\bX) w \equiv 0 \bmod \cP^2$.  
    \end{lemma}
    \begin{proof} 
       Recall that nontrivial quotients of a Dedekind domain are principal ideal rings \cite[p.278]{zariski2013commutative}.
       Let $\pi \in \cP\setminus \cP^2$. 
       Then, $\pi$ is a generator for the ideal $\cP/\cP^{2e}$ in $R/\cP^{2e}$.
       
       By definition of $\sQ_{e-1}$, every vector in $ \sQ_{e-1} \subseteq \bbF_{\cP}^n$ is the reduction of some $v\in R^n$ with $\pi^{e-1} v  \in \operatorname{Im}_{R}(\cL \bQ) + (\cP^{e})^{n}$. 
       In fact, using that $\cP^{e}/\cP^{2e}$ is generated by $\pi^{e}$, we may assume without loss of generality that $\pi^{e-1} v \in     \operatorname{Im}_{R}(\cL \bQ) + (\cP^{2e})^n$ by replacing $v$ with $v-\pi \varepsilon$ for some suitable error term $\varepsilon \in R^n$.

       Note that $f(\bX) v \equiv 0 \bmod \cP$ implies that there exists $y\in R^n$ with 
       \begin{equation}
            f(\bX)  v \equiv  \pi y \bmod \cP^{2e}. \label{eq:QuickBeetle}
       \end{equation}
       By Lemma \ref{lem: Restrict}, we have that $f(\bX) (\pi^{e-1} v)  \in  \operatorname{Im}_R(\cL \bQ) + (\cP^{2e})^n$. 
       Hence, $\pi^{e}y \in    \operatorname{Im}_R(\cL \bQ) + (\cP^{2e})^n$.  
       In particular, Lemma \ref{lem: IsotropyImQ} hence implies that   
       \begin{equation}
         \pi^{e} (q^{\T}y)  \equiv 0 \bmod \cP^{2e} \  \textnormal{ for every }\  q \in\operatorname{Im}_{R}(\cL \bQ) + (\cP^{e})^n. \label{eq:SillyCamel} 
       \end{equation}
        By definition of $\pi$ and the ideal arithmetic in Proposition \ref{prop: IdealArithmetic}, it holds that $\pi^{e} r \in \cP^{2e}$ for $r\in R$ if and only if $r\in \cP^{e}$.
        Using this twice in \eqref{eq:SillyCamel} yields that $\tilde{q}^{\T}y \equiv 0 \bmod \cP$ for every $\tilde{q}\in R^n$ with $\pi^{e-1} \tilde{q} \in \operatorname{Im}_R(\cL\bQ) + (\cP^{e})^n$.   
        Hence, by definition of $\sQ_{e-1}$, 
        \begin{equation}
            \tilde{q}^{\T}y \equiv 0 \bmod \cP\ \textnormal{ for every }\ \tilde{q}\in \sQ_{e-1}.
        \end{equation}
        In particular, using the assumption \eqref{eq:OilyHotel}, we may conclude that 
        \begin{equation}
             y \in \ker_{\cP}(f(\bX))^{\perp}.    \label{eq:WeavingClaw}
        \end{equation}
        
        Using that $\bX$ is symmetric, we have that $
            \operatorname{ker}_{\cP}(f(\bX))^{\perp }= \operatorname{Im}_{\cP}(f(\bX))$. 
        It hence follows from \eqref{eq:WeavingClaw} that there exists $\gamma \in R^n$ with $y \equiv f(\bX) \gamma \bmod \cP$. 
        Let $w \de v - \pi \gamma$ and note that we then have $f(\bX) w \equiv 0 \bmod \cP^2$ due to \eqref{eq:QuickBeetle}.           
    \end{proof}
    
    \begin{corollary}\label{cor: HigherP}
        Assume that $\cP^e\mid \cL$. 
        Then,  $ \cP^{\dim_{\cP}(\sQ_k) +1} \mid \Delta_{\bX} R$ for every $k\leq e-1$. 
    \end{corollary}
    \begin{proof} 
        Recall that $\psi_{k}$ denotes the characteristic polynomial of the action of $\bX$ on $\sQ_k$. 
        Recall from Lemma \ref{lem: Sq} that  $\psi_{k}^2\mid \phi_{\bX}$. 
        Consequently, $\psi_k \mid \phi_{\bX}'$ and hence  
           \begin{equation}
           \operatorname{ker}_{\cP}\bigl( \psi_k(\bX)\bigr) \subseteq \ker_{\cP}\bigl(\phi_{\bX}'(\bX)\bigr). \label{eq:GoodMusic} 
           \end{equation}
            Note that it follows from $\psi_k$ being the characteristic polynomial for $\sQ_k$ that $\ker_{\cP}(\psi_k(\bX) )$ has dimension no less than $\dim_{\cP}(\sQ_k)$. 

            If the inclusion \eqref{eq:GoodMusic} is strict or if $\ker_{\cP}(\psi_k(\bX) )$ has dimension strictly greater than $\dim_{\cP}(\sQ_k)$, then it follows that $\ker_{\cP}(\phi_{\bX}'(\bX))$ has dimension greater than or equal to $\dim_{\cP}(\sQ_k) + 1$. 
           Hence, by applying Lemma \ref{lem: Smith} to a basis of $\operatorname{ker}_{\cP}(\phi_{\bX}'(\bX))$, we indeed have $ \cP^{ \dim_{\cP}(\sQ_k) +1} \mid \Delta_{\bX} R $ in this case. 

            Now suppose that $ \operatorname{ker}_{\cP}( \psi_k(\bX)) =  \ker_{\cP}(\phi_{\bX}'(\bX))$ and that this vector space has dimension $\dim_{\cP}(\sQ_k)$.
            Then, considering that $\sQ_k \subseteq \operatorname{ker}_{\cP}( \psi_k(\bX))$, we have that $\operatorname{ker}_{\cP}( \phi_{\bX}'(\bX)) = \sQ_k$.   
            Using that $\cP^{k+1} \mid \cL$, Lemma \ref{lem: LiftSolutionSquared} now allows lifting a basis of $\operatorname{ker}_{\cP}(\phi_{\bX}'(\bX))$ to a system of independent solutions to $\phi_{\bX}'(\bX) w \equiv 0 \bmod \cP^2$. 
            Hence, $\cP^{2\dim_{\cP}(\sQ_k)} \mid \Delta_{\bX}R$ by Lemma \ref{lem: Smith}.  
            In particular, $\cP^{\dim_{\cP}(\sQ_k)+1} \mid \Delta_{\bX}R$. 
    \end{proof}
    \begin{proof}[Proof of Theorem \ref{thm: MainPlevelP2disc}]
        Immediate from Corollary \ref{cor: HigherP} with $e=1$. 
        (Recall from Lemma \ref{lem: NonEmpty} that $\sQ_0 \neq 0$.)  
    \end{proof}

    \section{Matrices over number fields --- Proofs for Section \ref{sec: AlgebraicNumber}}\label{sec: ProofsNumber}
    We here adopt the setting of Section \ref{sec: AlgebraicNumber}. 
    Thus, $K$ is a number field with ring of integers $\sO_K$, and we consider some symmetric $\bX \in \mathbb{Z}^{n\times n}$.  
    \subsection{Sufficient conditions --- Proof of Corollaries \ref{cor: Ramified1} and \ref{cor: Ramified2}}
    
    It suffices to establish Corollary \ref{cor: Ramified1}, since Corollary \ref{cor: Ramified2} is a direct consequence. 

    \begin{proof}[Proof of Corollary \ref{cor: Ramified1}]  
        We trivially have $p^2 \mid \Delta_{\bX}$ for every prime if $\Delta_{\bX} = 0$, so suppose that $\Delta_{\bX} \neq 0$.
        Let $\Delta_{\bX} = \pm \prod_{i=1}^m p_i^{h_i}$ be the prime factorization. 
    
        Consider a nonzero prime ideal $\cP \subseteq \sO_K$ with $\cP \mid \cL$.
        Then, $\cP^2 \mid \Delta_{\bX}R$ by Theorem \ref{thm: MainPlevelP2disc}. 
        The inclusion rules from Proposition \ref{prop: IdealArithmetic} imply that $\cP \mid p\sO_K$ for a prime $p$ if and only if $p\in \cP$. 
        Thus, there is a unique prime number with $\cP \mid p\sO_K$, this being the prime above which $\cP$ lies. 
        
        Using that $\cP^2 \mid \Delta_{\bX}\sO_K$ and the multiplication rules for ideal arithmetic in Proposition \ref{prop: IdealArithmetic} now yield $\cP^2 \mid (p_i \sO_K)^{h_i}$ for some  $p_i$ in the prime factorization. Then, $\cP^2 \mid p_i\sO_K$ or $h_i \geq 2$. 
    \end{proof}
    
    \subsection{Orthogonal matrices with a given level --- Proof of Proposition \ref{prop: TotallyRealSigned}}

    \begin{lemma}\label{lem: TotallyRealSigned}
        Assume that $K$ is totally real. 
        Then, every matrix $\bQ \in \sO_K^{n\times n}$ with $\bQ ^{\T}\bQ = \bI$ is a signed permutation. 
    \end{lemma}
    \begin{proof}
        Recall that the \emph{norm} of $k\in K$ is $\cN_{K/\bbQ}(k) \de \prod_{\sigma} \sigma(k)$ where the product runs over all embeddings $\sigma:K\to \bbC$. 
        It is classical that $\cN_{K/\bbQ}(r) \in \bbZ$ for every algebraic integer $r\in \sO_K$ \cite[Chapter 2]{marcus1977number}. 
    
        Every row $q = (q_1,\ldots,q_n)$ of $\bQ$ satisfies $\sum_{i=1}^n q_i^2 = 1$.         
        Hence, using that $K$ is totally real, we have $0 \leq \sigma(q_i)^2 \leq 1$ for every embedding $\sigma$. 
        Then, $\cN_{K/\bbQ}(q_i) \in \bbZ$ implies that $\lvert \sigma(q_i) \rvert=1$ if $q_i \neq 0$.   
        A theorem by Kronecker \cite{greiter1978simple} however states that the only algebraic integers $\alpha\in \sO_K$ with $\lvert \sigma(\alpha) \rvert = 1$ for all embeddings $\sigma:K\to \bbC$ are roots of unity. 
        Here, since $K \subseteq \bbR$, the only roots of unity are $\pm 1$. 
        
        We conclude that every row $q$ of $\bQ$ has exactly one nonzero entry $\pm 1$. 
        The orthogonality of the rows now yields that $\bQ$ is a signed permutation. 
    \end{proof}
    
    \begin{lemma}\label{lem: FinitelyMany}
        Assume that $K$ is totally real.
        Then, for every nonzero ideal $\cL \subseteq \sO_K$, there are finitely many matrices $\bQ \in K^{n\times n}$ with $\bQ^\T\bQ  = \bI$ and $\ell \bQ \in \sO_K^{n\times n}$ for every $\ell \in \cL$.   
    \end{lemma}
    \begin{proof}
        Fix some arbitrary nonzero $\ell \in \cL$ and let $\bG = \ell \bQ$. 
        Then, every column $g = (g_1,\ldots,g_n)$ of $\bG$ satisfies $\sum_{i=1}^n g_i^2 = \ell^2$. 
        Hence, using that $K$ is totally real, there exists some $C>0$ depending only on $\ell$ with $\lvert \sigma(g_i) \rvert \leq C$ for every $g_i$.

        Recall that the \emph{Minkowski embedding} $j:K\to \bbR^r$ of $K$ into Euclidean space sends every $k\in K$ to the vector $(\sigma(k))_{\sigma}$ \cite[C1,\S 5]{neukirch1999algebraic}. 
        Further, recall that this embedding sends $\sO_K$ into a lattice \cite[Proposition 5.2]{neukirch1999algebraic}. 
        In particular, $j(\sO_K)$ has a finite intersection with the compact box $\prod_{\sigma}[-C,C]$.
        This proves that there are only finitely many options for the entries of $\bG$, and hence also for $\bQ = \bG/\ell$.       
    \end{proof}
    
    \begin{proof}[Proof of Proposition \ref{prop: TotallyRealSigned}]
        Combine Lemmas \ref{lem: TotallyRealSigned} and \ref{lem: FinitelyMany}.
    \end{proof}
    
    \section{Algorithmic search for cospectral mates}\label{sec: Algorithm}
    Throughout this section, we assume that $\bX\in \bbZ^{n\times n}$ is a symmetric matrix over the (rational) integers and that $K$ is a totally real algebraic number field. 
    We think of $\bX$ as being potentially high-dimensional while $K$ is small and fixed. 
    \subsection{Algorithm outline}\label{sec: AlgorithmDescription}
    The main steps of the method are as follows; 
    see also Algorithm \ref{alg:orth_matrices} for pseudocode. 
    Details are given in the following Section \ref{sec: DetailedAlgDescription}. 
    \begin{description}[leftmargin = 1em]
        \item[Step 1. Determine potential ideal divisors] Compute the discriminant and all primes in the prime factorization $\Delta_{\bX} = \pm p_1^{h_1}p_2^{h_2}\cdots p_m^{h_m}$ for which $h_i \geq 2$ or that ramify in $K$ and have $h_i \geq 1$. 
        Subsequently, determine all prime ideals $\cP \subseteq \sO_K$ lying above the $p_i$ subject to the additional constraint that $\cP^2 \mid p_i \sO_K$ if $h_i =1$. 
        By Theorem \ref{thm: MainPlevelP2disc}, all prime ideals dividing the level $\cL \subseteq \sO_K$ of any orthogonal matrix $\bQ\in K^{n\times n}$ with $\bQ^{\T} \bX \bQ \in \sO_K^{n\times n}$ are of this form.  
        \item[Step 2. Upper-bound exponents] 
        For every prime ideal $\cP$ from the previous step, we determine some $E_{\cP}\geq 0$ with $\cP^e \nmid \cL$ for every $e> E_{\cP}$ as follows:   
        \begin{description}
             \item[Step 2a. Compute square root $\Gamma$] 
        The (approximate) square root polynomial $\Gamma\in \bbF_{\cP}[x]$ from \eqref{eq: Def_Gamma} can be computed without having to factorize $\phi_{\bX}\bmod \cP$ into irreducible factors by using a square-free factorization instead.  
        \item[Step 2b. Determine contractive and isotropic Krylov spaces] 
        Let $\sW_{\cP}(e)$ be the set of all $w\in (\sO_K/\cP^e)^n$ with $w\not\equiv 0 \bmod \cP$ that are representatives of a vector in $\sO_K^n$ for which $\sK_{\bX}(w)$ is $\Gamma$-contractive and isotropic up to the level $\cP^{e}$.
        (We do not track representatives in $\sO_K/\cP^{2e}$ and isotropy of level $\cP^{2e}$ here for algorithmic reasons.)
        We give a method to lift vectors from $\sW_{\cP}(e)$ to vectors in $\sW_{\cP}(e+1)$ by solving linear equations. 
        Initialization exploits that $w\bmod \cP$ lies in $\ker_{\cP}(\Gamma(\bX))$, which is low-dimensional.
        \item[Step 2c. Additional isotropy checks]  The previous step did not enforce the constraint of isotropy up to level $\cP^{2e}$ from Theorem \ref{thm: LevelBound}, only up to level $\cP^e$. 
        We enforce an additional linear condition and determine the subset $\bbW_{\cP}(e) \subseteq \sW_{\cP}(e)$ of those vectors that admit a lift that is isotropic up to level $\cP^{2e}$.    
        \end{description} 
        Theorem \ref{thm: LevelBound} now implies that $E_{\cP}  \de \inf\{e\geq 0: \mathbb{W}_{\cP}(e+1)  = \emptyset   \}$ has the desired property that  $\cP^e \nmid \cL$ for every $e> E_{\cP}$. 
        \item[Step 3. Construct orthogonal matrices] 
        We determine all orthogonal matrices $\bQ\in K^{n\times n}$ with $\bQ^{\T} \bX \bQ \in \sO_K^{n\times n}$ as follows:  
        \begin{description}
            \item[Step 3a. Candidate residues for column vectors]
            Each column $q$ of $\bQ$ has level ideal $\cL(q)\de \{r\in \sO_K: rq\in \sO_K^n \}$ of the form $\cL(q) = \prod_{\cP} \cP^{e_{\cP}}$ for $e_{\cP} \leq E_{\cP}$. 
            We iterate over the possible $e_{\cP}$ and pick some element $\ell \in \cL(q)$ such the exponent of $\cP^{e_{\cP}}$ is preserved in $\ell \sO_K$. 
            The sets $\bbW_{\cP}(e_{\cP})$ then allow us to determine a finite number of possible residue classes for $\ell q \bmod \ell \sO_K$.  
            \item[Step 3b. Lift candidate vectors] 
            For every fixed residue class $W$ there are only finitely many vectors $V\in \sO_K^n$ with $V\equiv W \bmod \prod_{\cP} \cP^{e_{\cP}}$ and $V^{\T}V = \ell^2$. 
            We construct all these vectors $V$ by exploiting that the solutions typically require many coordinates of $V$ to be zero. 
            Considering $V/\ell$ then yields a finite set $\sQ$ of possible column vectors for $\bQ$.  
             \item[Step 3c. Orthogonal matrix] We construct $\bQ$ by picking a subset $\{q_1,\ldots,q_n \} \subseteq \sQ$ subject to orthogonality $q_i^{\T} q_j = \bb1\{i=j \}$ and $\bQ^{\T}\bX \bQ\in \sO_K^{n\times n}$. 
        \end{description} 
    \end{description}

    If it terminates, then the algorithm determines the complete set of orthogonal matrices $\bQ \in K^{n\times n}$ with $\bQ^{\T}\bX \bQ \in \sO_K^{n\times n}$.  
    Every signed permutation of the columns of $\bQ$ again gives a valid orthogonal matrix, so we leave this implicit to reduce the size of the output by a factor $2^n n!$.

    Regarding termination in finite time: it may be clear from the preceding description that the main potential problem lies in the computation of $E_{\cP}$, as this number could potentially be infinite. 
    All other tasks are finite problems.
    We prove in Section \ref{sec: ProofTerminationInFiniteTime} that $E_{\cP}$ is finite whenever $\Delta_{\bX} \neq 0$.
{
\small 
    \begin{algorithm}
\caption{}
\label{alg:orth_matrices}
\begin{algorithmic}[1] 
    
    \renewcommand{\algorithmicrequire}{\textbf{Input:}}
    \renewcommand{\algorithmicensure}{\textbf{Output:}}
    
    \Require 
    Symmetric integer matrix $\bX \in \mathbb{Z}^{n \times n}$, totally real algebraic number field $K$, 
    integer $M \geq 1$ with $p^2 \nmid \Delta_{\bX}$ for all $p > M$.
    
    \Ensure 
    All orthogonal matrices $\bQ \in K^{n \times n}$ with $\bQ^{\T} \bX \bQ \in \sO_K^{n \times n}$.

    \State Determine all $\cP \subseteq \sO_K$ with $\cP^2 \mid \Delta_{\bX}\sO_K$.
    \For{each $\cP$ with $\cP^{2}\mid \Delta_{\bX}\sO_K$} 
        \State Compute approximate square root $\Gamma$ for $\phi_{\bX} \bmod \cP$. 
        \State $\sW(1) \gets \bigl\{ w \in \ker_{\cP}(\Gamma(\bX))\setminus \{0 \} : \forall j \leq n-1, w^{\T}\bX^j w \equiv 0 \bmod \cP \bigr\}$  
        \State $\bbW(1) \gets \bigl\{ w \in \sW(1) : \exists W \equiv w\bmod \cP^2,\, \sK_\bX(W) \textnormal{ isotropic up to level }\cP^2  \bigr\}$
        \State $e\gets 1$ 
        \While{$\bbW(e) \neq \emptyset$}
            \State $e\gets e+1$
            \State $\bbW(e)  \gets \bigl\{w\in (\sO_K/\cP^e)^n: \exists W \equiv w \bmod \cP^e,  \sK_\bX(W)\textnormal{ is }\Gamma\textnormal{-contractive}$ 
            \Statex \hspace{5.5cm} $\textnormal{up to level }\cP^e \textnormal{ and isotropic up to level }\cP^{2e}  \bigr\}$
        \EndWhile
        \State $E_{\cP}  \gets \inf\{e\geq 0: \mathbb{W}_{\cP}(e+1)  = \emptyset   \}$
    \EndFor
    \State $\sQ \gets \emptyset$ 
    \For{each exponent sequence $\{e_{\cP}\leq E_{\cP}: \cP^2\mid \Delta_{\bX} \sO_K \}$}
        \State Find $\ell\in \prod_\cP \cP^{e_{\cP}}$ with $\cP^{e_{\cP}+1}\nmid \ell \sO_K$ for all $\cP$ with $\cP^2\mid \Delta_{\bX}\sO_K$ 
        \State Determine vectors $V\in \sO_K^n$ with $V\bmod\cP^{e_{\cP}}$ in $\mathbb{W}_{\cP}(e_{\cP})$ and $V^{\T}V = \ell^2$ 
        \State Add the vectors $V/\ell \in K^n$ to $\sQ$
    \EndFor   
    \State Construct matrices $\bQ$ with columns $q_1,\ldots,q_n  \in \sQ$ subject to  $q_i^{\T} q_j = \bb1\{i=j \}$
    \State Check if $\bQ^{\T}\bX \bQ\in \sO_K^{n\times n}$
\end{algorithmic}
\end{algorithm}
}

    \subsection{Detailed description}\label{sec: DetailedAlgDescription}
    \subsubsection{Details for Step 1 --- Prime ideal divisors}
    The primes $p\in \bbZ_{\geq 2}$ that ramify in $K$ are the prime divisors of the discriminant of the number field $\operatorname{disc}(\sO_K)$ \cite[C3, Thm.24 \& C4, Thm.34]{marcus1977number}. 
    Computing the number field discriminant is a classical task, although computing it and determining the prime divisors may be prohibitively expensive for large number fields. 
    This is one of the many times when Assumption \ref{ass: SmallNumberField} is used. 
    
    Considering the prime ideals above prime divisors $p\mid \Delta_{\bX}$ that are ramified or have $p^2 \mid  \Delta_{\bX}$ then allows us to find all $\cP \subseteq \sO_K$ with $\cP^2\mid \Delta_{\bX} \sO_K$.   
    
    \subsubsection{Details for Step 2a --- Computing the square root}
    
    If no prime ideals $\cP\subseteq \sO_K$ with $\cP^2\mid \Delta_{\bX}\sO_K$ exist, then Corollary \ref{cor: Ramified2} is applicable and we may conclude that the only orthogonal matrices $\bQ \in K^{n\times n}$ with $\bQ^{\T} \bX \bQ \in \sO_K^{n\times n}$ are signed permutations. 
    Now assume that we are given some $\cP$ with $\cP^2\mid \Delta_{\bX}\sO_K$. 
    
    Recall that the polynomial $\Gamma \in \bbF_{\cP}[x]$ was defined in \eqref{eq: Def_Gamma} using the irreducible factorization of $\phi_{\bX} \bmod \cP$. 
    For computational purposes, it is more efficient to consider the \emph{square--free factorization} \cite{musser1971algorithms,yun1976square}: 
    \begin{equation}
        \textstyle \phi_{\bX}(x) \equiv \prod_i g_i(x)^{i} \bmod \cP, 
    \end{equation}
    where $g_1,g_2,\ldots \in \bbF_{\cP}[x]$ are coprime monic polynomials.  
    Then, 
    \begin{equation}
        \textstyle \Gamma(x) = \prod_{i} g_i(x)^{\lfloor i/2\rfloor}.
    \end{equation} 
    By Lemma \ref{lem: DiscriminantField}  and the assumption that $\cP^2 \mid \Delta_{\bX}\sO_K$, we know that $\deg(\Gamma) \geq 1$. 
    (This application of Lemma \ref{lem: DiscriminantField} uses that $\bbF_{\cP}$ is a finite field and hence perfect.)
    
    Note that the embedding $\bbZ \to \sO_K$ induces an inclusion of finite fields $\bbZ/p\bbZ \to \bbF_{\cP}$. 
    Using that an extension of finite fields is separable, it hence suffices to compute $\Gamma$ over $\bbZ/p\bbZ$.
    The polynomial over $\bbF_{\cP}$ then follows using the field inclusion.

    \subsubsection{Details for Step 2b --- Krylov spaces}
    Recall that $\sW_{\cP}(e) \subseteq (\sO_K/\cP^e)^n$ consists of the vectors that are representative of some $w\in \sO_K^n$ with $w\not\equiv 0 \bmod \cP$ and such that $\sK_{\bX}(w)$ is a $\Gamma$-contractive and isotropic subspace up to level $\cP^e$.

    We have $w\in \sW_{\cP}(1)$ if and only if $\Gamma(\bX)w \equiv 0 \bmod \cP$ and $w^{\T}\bX^j w \equiv 0 \bmod \cP$ for every $j\leq n-1$.    
    A linear-algebraic computation over $\bbZ/p\bbZ$ yields a basis $k_1,\ldots,k_d \in \ker_{\cP}(\Gamma(\bX))$. 
    Let $\bK = [k_1,\ldots,k_d] $ be the $n\times d$ matrix with these columns. 
    Then, $\Gamma(\bX)w \equiv 0 \bmod \cP$ if and only if $w \equiv \bK v \bmod \cP$ for some $v\in \bbF_{\cP}^d$, and  $w^{\T}\bX^j w \equiv 0 \bmod \cP$ for all $j\leq n-1$ if and only if 
    \begin{equation}
        v^{\T} (\bK^{\T} \bX^j \bK) v = 0, \quad \forall j\leq \operatorname{deg}(\Gamma)-1.\label{eq:ColdFog}
    \end{equation}
    We here used that $\bK v\in \ker_{\cP}(\Gamma(\bX))$ to restrict the powers to $j\leq \operatorname{deg}(\Gamma)-1$.  

    Note that \eqref{eq:ColdFog} is an intersection of $\deg(\Gamma)$ quadrics. 
    The general problem of solving an intersection of multivariate quadrics is known as the \emph{MQ problem} in cryptography and is expensive in large problems \cite[\S2.3.1]{ding2006multivariate}.
    We are however concerned with cases where $\operatorname{deg}(\Gamma)$ is small; recall Assumption \ref{ass: DegGamma}.
    Thus, solving \eqref{eq:ColdFog} enables us to determine the following set: 
    \begin{equation}
        \sW_{\cP}(1) = \{ w\in \ker_{\cP}(\Gamma(\bX))\setminus \{0 \} :  \forall j\leq n-1,\,  w^{\T}\bX^j w = 0 \}. 
    \end{equation}

    \begin{remark}
        If $\cP$ lies above $p=2$, then $\bbF_{\cP}$ has characteristic $2$ and hence 
        \begin{equation}
            \textstyle v^{\T}(\bK^{\T} \bX^j \bK) v \equiv  \sum_{i=1}^d D_i^{\bc{j}} v_i^2 \equiv \bigl(\sum_{i=1}^dD_i^{\bc{j}} v_i\bigr)^2\bmod \cP, 
        \end{equation}
        where $D_{i}^{\bc{j}}$ is the $i$-th diagonal entry of $\bK^{\T} \bX^j \bK$ and we used that $D_i^{\bc{j}} = (D_i^{\bc{j}})^2$ since this is an element of the subfield $\bbZ/2\bbZ\subseteq \bbF_{\cP}$.
        Thus, the quadratic constraints \eqref{eq:ColdFog} could be replaced linear ones in this case. 
    \end{remark}
    Now suppose that we are given some $w\in  \sW_{\cP}(e)$ for $e\geq 1$. 
    Our task is to find the lifts to $\sW_{\cP}(e+1)$, or to show that no such lifts exist.

    Fix some arbitrary $\tilde{w}\in (\sO_K/\cP^{e+1})^n$ with $\tilde{w} \equiv w \bmod \cP^e$, and consider some $\pi \in \cP$ with $\pi \not\in \cP^{2}$. 
    Then, $\sO_K/\cP^{e+1}$ is a principal ideal ring with ideals given by the reductions of $\{\cP^j:j\leq e \}$ and $\cP/\cP^{e+1}$ is generated by $\pi \bmod \cP^{e+1}$. 
    Hence, every $\omega \in (\sO_K/\cP^{e+1})^n$ with $\omega \equiv w \bmod \cP^e$ can be uniquely represented as
    \begin{equation}
        \omega \equiv \tilde{w} + \pi^e \delta \bmod \cP^{e+1}, \label{eq:MadFace} 
    \end{equation}
    for $\delta \in \bbF_{\cP}^n$. 
    We next determine when the Krylov space is $\Gamma$-contractive and isotropic. 

    Let $Q_e(x)\in (\sO_K/ \cP^{e})[x]$ be a polynomial of minimal degree with
    \begin{equation}
        Q_e(\bX)w \equiv 0 \bmod \cP^e.\label{eq:YoungPug} 
    \end{equation}
    Such a polynomial can be computed by linear-algebraic computations expressing that $\bX^d w$ is in the span of $\{\bX^j w: j<d \}$ over $\sO_K/\cP^e$. 
    Let  $\tilde{Q}_e\in \sO_K[x]$ be an arbitrary monic lift of $Q_e$ and consider the unique vector $r\in \bbF_{\cP}^n$ with 
        \begin{equation}
            \tilde{Q}_e(\bX) \tilde{w} \equiv \pi^e r \bmod \cP^{e+1}.  
        \end{equation}
    The following Lemmas \ref{lem: ContractiveLift} and \ref{lem: IsotropicLift} turn the constraints of $\Gamma$-contractivity and isotropy into a system of linear equations. 
    \begin{lemma}\label{lem: ContractiveLift}
        With notation and assumptions as above, $\omega = \tilde{w} + \pi^e \delta$ satisfies that $\sK_{\bX}(\omega)$ is $\Gamma$-contractive up to level $\cP^{e+1}$ if and only if  
        \begin{equation}
            \Gamma(\bX) Q_e(\bX) \delta \equiv - \Gamma(\bX)r \bmod \cP.  \label{eq:TediousBall} 
        \end{equation}
    \end{lemma}
    \begin{lemma}\label{lem: IsotropicLift}
        Let $c_j\in \bbF_{\cP}$ be the unique scalars with $\tilde{w}^{\T} \bX^j \tilde{w} \equiv \pi^{e} c_j$ for every $j\leq n-1$. 
        Then,   $\omega = \tilde{w} + \pi^e \delta$ satisfies that $\sK_{\bX}(\omega)$ is isotropic up to level $\cP^{e+1}$ if and only if $\delta$ satisfies the following linear equations over $\bbF_{\cP}$:
        \begin{equation}
            2 w^{\T} \bX^j \delta  \equiv - c_j \bmod \cP \ \ \textnormal{ for } \ \ j=0,1,\ldots,\operatorname{deg}(Q_e) + \operatorname{deg}(Q_1)-1.\label{eq:WeavingBall} 
        \end{equation} 
    \end{lemma}  
    
    The proofs of Lemmas \ref{lem: ContractiveLift} and \ref{lem: IsotropicLift} are given in Sections \ref{sec: LinearRephraseContractive} and \ref{sec: LinearIsotropyProofs}, respectively.
    By solving the linear equations \eqref{eq:TediousBall} and \eqref{eq:WeavingBall}, we can determine all vectors in $\sW_{\cP}(e+1)$ that are lifts of some given $w\in \sW_{\cP}(e)$, or conclude that no such lifts exists.   

    \begin{remark}
        Note that $w\in \sW_{\cP}(e)$ if and only if $rw \in \sW_{\cP}(e)$ for every $r\in \sO_{K}/\cP^e$ with $r\not\equiv 0 \bmod \cP$.
        This means that $r$ reduces to a unit. 
        It hence suffices to determine a single representative in $(\sO_K/\cP^{e})^n$ up to scalar multiplication. 
        Our implementation exploits this to reduce repeated computations by tracking only the vector with $w_{i_*} \equiv 1 \bmod \cP^{e}$ for the smallest index $i_* \leq n$ with $w_{i_*} \not\equiv 0 \bmod \cP$.
        At the end of Step 2, we finally consider all scalar multiples of the found vectors.  
    \end{remark}
   \begin{remark}
       If two prime ideals $\cP_1,\cP_2 \subseteq \sO_K$ are in the same orbit of the Galois group, \ie there exists $\sigma \in \operatorname{Gal}(K/\bbQ)$ with $\sigma(\cP_1) = \cP_2$, then naturally $\sigma(\sW_{\cP_1}(e)) = \sW_{\cP_2}(e)$.
       This could be used to reduce redundant computations, as it implies that it suffices to do the computations only once for each orbit.
       In particular, if $K$ is the splitting field of a polynomial (\ie a Galois extension), then there is only a single orbit \cite[C1,\S9]{neukirch1999algebraic}.  
   \end{remark}

    \subsubsection{Details for Step 2c --- Additional isotropy checks} 
    Recall that Theorem \ref{thm: ColumnRestriction} involves isotropy up to level $\cP^{2e}$, not only up to level $\cP^{e}$.
    The reason why we nonetheless only tracked representatives over $\sO_K/\cP^e$ is due to the combinatorial explosion in the number vectors over $\sO_K/ \cP^{2e}$ relative to $\sO_K/\cP^e$ when $n$ is large.  

    The following variant on Lemma \ref{lem: IsotropicLift} allows us to filter out the subset $\mathbb{W}_{\cP}(e) \subseteq \sW_{\cP}(e)$ consisting of those vectors that admit a lift that is isotropic up to level $\cP^{2e}$. 
    Crucially, we avoid computing these lifts explicitly. 
    Recall that $\pi \in \cP\setminus \cP^2$.  
    \begin{lemma}\label{lem: StrongIsotropicLift}
        Consider some $w\in \sW_{\cP}(e)$.
        Pick some lift $W \in (\sO_K/\cP^{2e})^n$ with $W  \equiv w \bmod \cP^e$ and let $C_j\in \sO_K/\cP^e$ be the scalars with $W^{\T} \bX^j W \equiv \pi^{e} C_j \bmod \cP^{2e}$. 
        Then, there exists  $\Omega\in \sO_K^n$ with $\Omega \equiv w \bmod \cP^e$ and  $\sK_{\bX}(\Omega)$ isotropic up to level $\cP^{2e}$ if and only if the following system  admits a solution $\delta \in (\sO_K/\cP^e)^n$: 
        \begin{equation}
            2 w^{\T} \bX^j \delta  \equiv - C_j \bmod \cP^{e}  \ \ \textnormal{ for } \ \ j=0,1,\ldots,2\operatorname{deg}(Q_e) -1. 
        \end{equation}
    \end{lemma} 

    \subsubsection{Step 3a. --- Candidate residues for column vectors}
    Recall we associate a level ideal $\cL(q) = \{r\in \sO_K : rq \in \sO_K^n \}$ to every column vector $q\in K^{n}$ of the potential orthogonal matrix $\bQ$. 
    This ideal has the following form: 
    \begin{equation}
        \textstyle \cL(q) = \prod_{\cP} \cP^{e_{\cP}} \ \textnormal{ for certain }e_{\cP} \leq E_{\cP}.
    \end{equation} 
    Here, $E_{\cP}  = \inf\{e\geq 0: \mathbb{W}_{\cP}(e+1)  = \emptyset   \}$. 
    We iterate over the possible exponents $e_{\cP}$, and our goal is to determine the corresponding possibilities for $q$. 
    
    If all exponents are zero, then $q$ must be a vector with a single nonzero entry $\pm 1$; recall the proof of Lemma \ref{lem: TotallyRealSigned}.
    We next consider the nontrivial case with $e_{\cP} \geq 1$ for some $e_{\cP}$. 
    Then, by the Chinese remainder theorem over $\sO_K/\prod_{\cP} \cP^{e_{\cP}+1}$, there exists some $\ell\in \cL(q)$ with 
    \begin{equation}
        \ell \in \cP^{e_{\cP}}\setminus \cP^{e_{\cP}+1} \ \textnormal{ for all }\ \cP\ \textnormal{ with }e_{\cP}\geq 1. \label{eq:JustNose}
    \end{equation}
    This element is not unique and the efficiency of subsequent steps will depend on the choice. 
    Our implementation picks $\ell$ such that it has minimal absolute norm \cite[p.8]{neukirch1999algebraic} subject to \eqref{eq:JustNose}.
    In any case, $\ell \in \cL(q)$ implies that its ideal can be factorized as 
    \begin{equation}
        \ell \sO_K = \cL(q) \cJ.
    \end{equation}
    The ideal $\cJ \subseteq \sO_K$ will be nontrivial if $\cL(q)$ was not principal.
        
    It follows from \eqref{eq:JustNose} and the ideal arithmetic of Proposition \ref{prop: IdealArithmetic} that $\ell q \not\equiv 0 \bmod \cP$ whenever $e_{\cP}\geq 1$.
    It can here be deduced from Theorem \ref{thm: ColumnRestriction} that $\ell q \bmod \cP^{e_{\cP}}$ lies in $\bbW_{\cP}(e_{\cP})$.
    Further, the condition that $q$ has level ideal $\cL(q)$ implies that $\ell q \equiv 0 \bmod \cJ$.  
    We can hence determine a set of vectors $W\in (\sO_K/\ell \sO_K)^n$ that could plausibly satisfy $W \equiv \ell q \bmod \ell \sO_K$ by using the Chinese remainder theorem: 
    \begin{equation}
       \begin{cases}
            W \equiv w_\cP \bmod \cP^{e_{\cP}}\ \textnormal{ with }\ w_{\cP} \in \bbW_{\cP}(e_{\cP}) \textnormal{ whenever }e_{\cP}\geq 1, \\ 
            W\equiv 0 \bmod \cJ. 
       \end{cases}
       \label{eq:TipsyDen}  
    \end{equation}
    The set of vectors that arises from this system should typically be small relative to $(\sO_K/\ell \sO_K)^n$ and $(\sO_K/\cL(q))^n$, which are also finite but grow exponentially in $n$.

    \subsubsection{Details for Step 3b --- Lift candidate vectors}
    If $q$ is a column of an orthogonal matrix, then $(\ell q)^{\T} (\ell q) = \ell^2 (q^{\T}q) = \ell^2$.  
    Given some $W \in (\sO_K/\ell \sO_K)^n$ arising from \eqref{eq:TipsyDen} we hence wish to determine all vectors $V \in \sO_K^n$ with 
    \begin{equation}
        V \equiv W \bmod \ell \sO_K\ \textnormal{ and }\ V^T V = \ell^2. \label{eq:PaleTrunk}  
    \end{equation}
    We start by considering a constraint that often quickly shows that $W$ is such that no solutions to \eqref{eq:PaleTrunk} exist, or restricts the search space otherwise.

    Let $\sigma_1,\ldots,\sigma_r: K\to \bbR$ be the field embeddings of our totally real number field. 
    Let us define a \emph{$\ell$-twisted trace} $T_{\ell} :K\to \bbR$ and norm $\Vert \cdot \Vert_\ell$ by 
    \begin{equation}
        \textstyle T_\ell (x) \de \sum_{m} \sigma_m(x/\ell^2) \quad \text{ and }\quad \Vert x \Vert_\ell \de \sqrt{T_\ell(x^2)}.  \label{eq:TallRam}
    \end{equation}
    For every $w\in \sO_K$ representing a residue class in $\sO_K/\ell \sO_K$, we then compute 
    \begin{equation}
        \tau_\ell(w) \de \min_{x\in  \ell \sO_K} T_\ell\bigl((w - x)^2\bigr) =  \min_{y\in  \sO_K} T_1\bigl((w/\ell - y)^2\bigr), \label{eq:AvidActor} 
    \end{equation}
    and we store a value $x(w) \in \ell \sO_K$ achieving the minimum.   
    This is a \emph{closest vector problem} for the lattice $\sO_K$ and its tractability uses Assumption \ref{ass: SmallNumberField}.
    
    It follows from $V^{\T}V = \ell^2$ and linearity that $\sum_i T_\ell(v_i^2) = r$. 
    Thus, the existence of a solution to \eqref{eq:PaleTrunk} necessitates that  
    \begin{equation}
        \sum_{i=1}^n  \tau_\ell(w_i) \leq r. \label{eq:MildCity}
    \end{equation}
    If \eqref{eq:MildCity} is violated, then \eqref{eq:PaleTrunk} does not admit solutions.   
    Now suppose that \eqref{eq:MildCity} holds.
    Heuristically, for large $n$, this should imply that $W$ has only a few nonzero coordinates. 
    The coordinates that are  $0\bmod \ell \sO_K$ yield algebraic integer entries in $V/\ell$ and could hence only be $0$ or $\pm 1$ by $V/\ell$ being a unit vector; recall the proof of Lemma \ref{lem: TotallyRealSigned}.
    In fact, these lifts have to be zero due to $W$ also having nonzero coordinates.
    
    It hence remains to lift the nonzero coordinates, which we do by picking values one-by-one. 
    Suppose that the $v_j$ with $j<i$ have been assigned a value. 
    Then, we restrict the search space for $v_i$ using that completion to a full vector requires satisfaction of the following \emph{lookahead constraint}:
    \begin{equation}
        \textstyle T_\ell\bigl(v_i^2\bigr) \leq r - \sum_{j<i} T_\ell\bigl(v_j^2) - \sum_{j>i}  \tau_\ell(w_j). \label{eq:WeavingSwan}
    \end{equation}
    In particular, if we write $v_i = w_i - x(w_i) + \delta_i$ with $x(w_i)$ the minimizer for \eqref{eq:AvidActor} and $\delta_i \in \ell\sO_K$, then the triangle inequality and \eqref{eq:WeavingSwan} yield that
    \begin{equation}
        \Vert \delta_i \Vert_\ell  \leq   \sqrt{\tau_\ell(w_i)} + \sqrt{r - \sum_{j<i} T_\ell\bigl(v_j^2) - \sum_{j>i}  \tau_\ell(w_j)}. \label{eq:SlowCup}
    \end{equation}
    We compute all $\delta_i$ satisfying \eqref{eq:SlowCup} and additionally filter the choices using that $v_i$ must satisfy the following \emph{spent budget constraints}:  
    \begin{equation}
        \textstyle \sigma_m\bigl(v_i\bigr)^2 \leq \sigma_m\bigl(\ell^2\bigr) - \sum_{j<i}\sigma_m\bigl(v_j^2\bigr) \textnormal{ for every }m\leq r.   
    \end{equation}
    Once a full vector is composed, we finally check if equality holds in $V^{\T}V = \ell^2$. 
   
    \subsubsection{Details for Step 3c --- Orthogonal matrix}
    Considering the vectors $V/\ell$ with $V$  satisfying \eqref{eq:PaleTrunk}, and the standard basis vectors $\pm e_i $ corresponding to $\cL(q)=\sO_K$, yields a finite set $\sQ \subseteq K^n$ of potential column vectors.
    
    It remains to determine orthogonal subsets of size $n$. 
    We do this by exploiting that an orthogonal set of vectors $\{q_i: i=1,\ldots,m \}$ can be completed to an orthogonal basis by adding standard basis vectors if and only if the number of active coordinates is exactly $m$. 
    Here, $j\in \{1,\ldots,n \}$ is \emph{active} if there exists some $q_i$ whose $j$th entry is nonzero. 
    Recursively selecting such subsets based on the vectors $V/\ell$ turns out to run sufficiently quickly in practice.  

    Given an orthogonal set of vectors $\{q_i:i=1,\ldots,n \}$ we then finally check if the associated orthogonal matrix $\bQ\in K^{n\times n}$ indeed satisfies $\bQ^{\T} \bX \bQ\in \sO_K^{n\times n}$. 
    The algorithm returns all orthogonal matrices over $K$ with the latter property, up to signed permutation of the columns.

    \subsection{Proof of algorithmic properties} 
    \subsubsection{Termination in finite time --- Proof of Theorem \ref{thm: Algorithm}}\label{sec: ProofTerminationInFiniteTime}
    Let $q_k \in \bbZ[w_1,\ldots,w_n]$ be the polynomial defined by the quadratic form associated to $\bX^k$:  
    \begin{equation}
        q_k(w_1,\ldots,w_n) = w^{\T} \bX^k w \quad \textnormal{ with }\quad w = (w_1,\ldots,w_n)^{\T}.   
    \end{equation}
    
    \begin{lemma}\label{lem: TrivialSolution}
        Assume that $\Delta_{\bX} \neq 0$. 
        Then, the only $w\in \bbC^n$ satisfying the equations $q_k(w) = 0$ for all $k\leq n-1$ is the trivial solution $w= 0$.  
    \end{lemma}
    \begin{proof}
        Let $\lambda_1,\ldots,\lambda_n\in \bbR$ be the eigenvalues of the symmetric matrix $\bX \in \bbZ^{n\times n}$ and consider an associated orthonormal basis of eigenvectors $U_1,\ldots,U_n \in \bbR^n$, so that $\bX = \sum_{i} \lambda_i U_iU_i^{\T}$. 
        Write $w = \sum_{i=1}^n c_i U_i$ for arbitrary coefficients $c_i \in \bbC$. 
        Then,   
        \begin{equation}
            q_k(w) = 0 \quad \iff \quad \sum_{i=1}^n c_i^2\lambda_i^k = 0.  \label{eq:VagueImp}
         \end{equation}
        The system of equations $\{\sum_{i} x_i \lambda_i^k = 0: k \leq n-1\}$ is specified by a Vandermonde matrix with determinant $\prod_{i<j} (\lambda_{j} - \lambda_i) =  \sqrt{\lvert  \Delta_{\bX} \rvert}$. 
        Hence, using that $\Delta_{\bX} \neq 0$, the system only has the trivial solution. 
        Thus,  $c_i^2 = 0$ for all $i$ and hence $w = 0$. 
    \end{proof}
    \begin{lemma}
        Assume that $\Delta_{\bX} \neq 0$.
        Then, for every $i\leq n$ there exist $h_i \geq 0$ and $n_i \in \bbZ\setminus \{0 \}$ such that the monomial $n_i w_i^{h_i}$ can be expressed as a $\bbZ[w_1,\ldots,w_n]$-linear combination of the $q_k$. 
        That is, 
        \begin{equation}
            \textstyle n_i w_i^{h_i} = \sum_{k=0}^{n-1} s_k^{\bc{i}}(w_1,\ldots,w_n) q_k(w_1,\ldots,w_n),\label{eq:GiddyPig} 
        \end{equation}
        for certain polynomials $s_0^{\bc{i}},\ldots,s_{n-1}^{\bc{i}} \in \bbZ[w_1,\ldots,w_n]$. 
    \end{lemma}
    \begin{proof} 
        Using Lemma \ref{lem: TrivialSolution} and Hilbert's Nullstellensatz \cite[p.4]{hartshorne2013algebraic} yields that $w_i$ is an element of the radical of the ideal generated by the polynomials $q_k$ over $\bbC$. 
        That is, there exists $h_i \geq 0$ and polynomials $\tilde{s}_k^{\bc{i}} \in \bbC[w_1,\ldots,w_n]$ with  
        \begin{equation}
            \textstyle w_i^{h_i} = \sum_{k=0}^{n-1} \tilde{s}_k^{\bc{i}}(w_1,\ldots,w_n) q_k(w_1,\ldots,w_n). \label{eq:VividKnot}
        \end{equation}
        
        Moreover, it follows from the fact that $w_i^{h_i}$ and $q_k(w_1,\ldots,w_n)$ are polynomials over $\bbQ$ that the $\tilde{s}_k^{\bc{i}}$ can also be taken to have coefficients over $\bbQ$. 
        Indeed, one can view \eqref{eq:VividKnot} as a system of linear equations for the coefficients of the $\tilde{s}_k^{\bc{i}}$, and use that a system of linear equations has a solution over a field extension if and only if it has one in the original field. 
        Clearing denominators now yields \eqref{eq:GiddyPig}.          
    \end{proof}
    The following Lemmas \ref{lem: IsotropyXk} and \ref{lem: KrylovIsotropyBound} do not require $K$ to be totally real. 
    \begin{lemma}\label{lem: IsotropyXk}
        Assume that $\Delta_{\bX} \neq 0$ and consider a nonzero prime ideal $\cP \subseteq \sO_K$. 
        Then, there exists some $e \geq 0$ such that every $w\in \sO_K^n$ with $w^{\T}\bX^k w \equiv 0 \bmod \cP^{e}$ for every $k\leq n-1$ satisfies that $w \equiv 0 \bmod \cP$. 
    \end{lemma}
    \begin{proof}
        Let $e\geq 0$ be sufficiently large so that the integers $n_i$ in \eqref{eq:GiddyPig} satisfy that $n_i \not\in \cP^e$ for every $i$. 
        Note that \eqref{eq:GiddyPig} implies that $n_i w_i^{h_i} \equiv 0 \bmod \cP^e$ whenever $w^{\T}\bX^k w \equiv 0 \bmod \cP^e$ for every $k\leq n-1$.   
        By definition of $e$ and the ideal arithmetic Proposition \ref{prop: IdealArithmetic}, it is only possible to have $n_i w_i^{h_i} \in \cP^{e}$ if $w_i \in \cP$. 
    \end{proof}
    \begin{lemma}\label{lem: KrylovIsotropyBound}
        Assume that $\Delta_{\bX} \neq 0$ and consider a nonzero prime ideal $\cP \subseteq \sO_K$. 
        Then, there exists $e\geq 0$ depending on $\bX$ such that there does not exist any $w\in \sO_K^n$ with $w\not\equiv 0 \bmod \cP$ for which $\sK_{\bX}(w)$ is isotropic up to level $\cP^e$.  
    \end{lemma}
    \begin{proof}
        This is immediate from Lemma \ref{lem: IsotropyXk} since $\sK_{\bX}(w)$ being isotropic up to level $\cP^e$ necessitates that $w^\T\bX^k w \equiv 0 \bmod \cP^{e}$.
    \end{proof}
    \begin{corollary}
        Assume that $\Delta_{\bX} \neq 0$.  
        Then, Algorithm \ref{alg:orth_matrices} terminates in finite time. 
    \end{corollary}
    \begin{proof}
        The only potential problem was whether the number $E_{\cP} = \inf\{e\geq 0: \mathbb{W}_{\cP}(e+1) =\emptyset \}$ is finite. 
        Lemma \ref{lem: KrylovIsotropyBound} implies that this is indeed the case since  $\sW_{\cP}(e) = \emptyset$ implies that $\mathbb{W}_{\cP}(e) = \emptyset$. 
    \end{proof}
    
    \begin{remark}
        Note that the $\Gamma$-contractivity was not used to prove termination in finite time. 
        Its relevance is only for the efficiency of the algorithm.   
    \end{remark}

    \subsubsection{Linear rephrasing for  contractivity of lifts --- Proof of Lemma \ref{lem: ContractiveLift}}\label{sec: LinearRephraseContractive}
    Recall from \eqref{eq:YoungPug} that $\tilde{Q}_e\in \sO_K[x]$ is a monic polynomial of minimal degree with $\tilde{Q}_e(\bX) w\equiv 0 \bmod \cP^e$. 
    Further, recall that we fix some $\pi\in \cP$ with $\pi \not\in \cP^{2}$.    
    \begin{lemma}\label{lem: DivisionQe}
        Fix some $w\in \sO_K^n$ and $e\geq 1$. 
        Then, for every $q\in \sO_K[x]$ with 
        $
            q(\bX)w \equiv 0 \bmod \cP^e 
        $
        there exist polynomials $d(x),r(x) \in \sO_K[x]$ with 
        \begin{equation}
            q(x) \equiv d(x)\tilde{Q}_e(x) + \pi r(x) \bmod \cP^{e+1}.   
        \end{equation}
    \end{lemma}
    \begin{proof}
    Subtracting a multiple of $\tilde{Q}_e$ from $q$ ensures that $\operatorname{deg}(q)<\operatorname{deg}(\tilde{Q}_e)$. 
    Then, the leading coefficient must be a multiple of $\pi$ in $\sO_K/\cP^{e+1}$, as we could otherwise invert it to get a monic polynomial, contradicting the degree minimality of $\tilde{Q}_e$. 
    
    In this case, subtracting multiples of $ \pi Q_1$ ensures that $q(x) = \pi \gamma(x) Q_1(x) + \rho(x)$ for some $ \gamma(x),\rho(x)\in \sO_K[x]$ with $\deg(\rho)< \deg(Q_1)$. 
    However, we then have $\rho(\bX) w \equiv 0 \bmod \cP$ so that the degree minimality of $Q_1$ implies that $\rho \equiv 0 \bmod \cP$.
    Then, $\rho = \pi \rho' \bmod \cP^{e+1}$. 
    The claim then follows with $r(x) = \gamma(x) Q_1(x) + \rho'(x)$.    
    \end{proof}
    \begin{lemma}\label{lem: ContractiveQ}
        Fix some $w\in \sO_K^n$ for which $\sK_{\bX}(w)$ is $\Gamma$-contractive up to level $\cP^e$ for $e\geq 1$. 
        Consider some $\tilde{w} \in \sO_K^n$ with $\tilde{w}\equiv w\bmod \cP^e$. 
        Then, $\sK_{\bX}(\tilde{w})$ is $\Gamma$-contractive up to level $\cP^{e+1}$ if and only if $\Gamma(\bX) \tilde{Q}_e(\bX)\tilde{w} \equiv 0 \bmod \cP^{e+1}$.  
    \end{lemma}
    \begin{proof}
        By Definition \ref{def: Krylov} and the Cayley--Hamilton theorem, $\sK_{\bX}(\tilde{w})$ consists exactly of the vectors of the form $q(\bX)\tilde{w}$ for $q\in \sO_K[x]$. 
        In particular, $\tilde{Q}_e(\bX) \tilde{w}\in \sK_{\bX}(\tilde{w}) \cap \cP^e$ by definition of $\tilde{Q}_e$. 
        Thus, by Definition \ref{def: Contractive}, the Krylov space being contractive up to level $\cP^{e+1}$ certainly implies that $\Gamma(\bX) \tilde{Q}_e(\bX)\tilde{w} \equiv 0 \bmod \cP^{e+1}$. 

        It remains to show that $\Gamma(\bX) \tilde{Q}_e(\bX)\tilde{w} \equiv 0 \bmod \cP^{e+1}$ implies that $\Gamma(\bX) q(\bX) \tilde{w} \equiv 0 \bmod \cP^{e+1}$ whenever $q(x)\in \sO_K[x]$ is a polynomial with $q(\bX)\tilde{w} \equiv 0 \bmod \cP^e$. 
        By Lemma \ref{lem: DivisionQe} we can then write $q(x) \equiv d(x) \tilde{Q}_e(x) + \pi r(x) \bmod \cP^{e+1}$. 
        It then follows from $q(\bX) \tilde{w} \equiv 0 \bmod \cP^e$ that $ r(\bX) \tilde{w} \equiv 0 \bmod \cP^{e-1}$. 
        That $\sK_{\bX}(w)$ is $\Gamma$-contractive up to level $\cP^e$ then yields that $\Gamma(\bX) r(\bX) \tilde{w} \equiv 0 \bmod \cP^e$. 
        Hence,  
        \begin{equation}
            \Gamma(\bX)q(\bX) \tilde{w} \equiv d(\bX) \Gamma(\bX) \tilde{Q}_e(\bX) \tilde{w} + \pi\Gamma(\bX) r(\bX) \tilde{w} \equiv 0 \bmod \cP^{e+1},  
        \end{equation}
        as desired. 
    \end{proof}
    \begin{proof}[Proof of Lemma \ref{lem: ContractiveLift}]
        This is immediate from Lemma \ref{lem: ContractiveQ}, since the linear equality \eqref{eq:TediousBall} is equivalent to having $\Gamma(\bX) \tilde{Q}_e(\bX) \omega  \equiv 0 \bmod \cP^{e+1}$ for $\omega = \tilde{w} + \pi^e \delta$.  
    \end{proof}
    
    \subsubsection{Linear rephrasing for isotropy of lifts --- Proofs of Lemmas \ref{lem: IsotropicLift} and \ref{lem: StrongIsotropicLift}}\label{sec: LinearIsotropyProofs}
    The following result covers both of the desired results as a special case:
    \begin{lemma}
        Consider some $w\in \sO_K^n$ whose Krylov space $\sK_{\bX}(w)$ is isotropic up to level $\cP^e$ for $e\geq 1$.
        Fix some $k\in \{1,\ldots,e \}$ and $\pi \in \cP \setminus \cP^{2}$.   
        
        Pick an arbitrary lift $\tilde{w} \in (\sO_K/\cP^{e+k})^n$ with $\tilde{w}  \equiv w \bmod \cP^e$ and let $C_j\in \sO_K/\cP^k$ satisfy $\tilde{w}^{\T} \bX^j \tilde{w} \equiv \pi^{e} C_j \bmod \cP^{e+k}$. 
        Consider $\omega\in \sO_K^n$ and $\delta \in (\sO_K/\cP^k)^n$ with 
        \begin{equation}
            \omega \equiv  \tilde{w} + \pi^{e} \delta  \bmod \cP^{e+k}.\label{eq:CalmJar} 
        \end{equation}
        Then,  $\sK_{\bX}(\omega)$ is isotropic up to level $\cP^{e+k}$ if and only if $\delta$ satisfies the following equations: 
        \begin{equation}
            2 w^{\T} \bX^j \delta  \equiv - C_j \bmod \cP^{k}  \ \ \textnormal{ for } \ \ j=0,1,\ldots,\operatorname{deg}(Q_e) + \operatorname{deg}(Q_k)-1.\label{eq:KindVine} 
        \end{equation}
    \end{lemma}
    \begin{proof}
        Note that $\sK_{\bX}(\omega)$ is isotropic up to level $\cP^{e+k}$ if and only if $\omega^{\T} \bX^j \omega \equiv 0 \bmod \cP^{e+k}$ for all $j$. 
        Here, expanding \eqref{eq:CalmJar} using the symmetry of $\bX$, 
        \begin{align}
            \omega^{\T} \bX^j \omega &= \tilde{w}^{\T} \bX^j \tilde{w} +\pi^e 2\tilde{w}^{\T} \bX^j \delta + \pi^{2e}\delta^{\T}\bX^j \delta \equiv \pi^e \Bigl( C_j + 2w^{\T}\bX^j\delta\Bigr) \bmod \cP^{e+k},     
        \end{align}
        where we used that $\tilde{w} \equiv w \bmod \cP^e$, the assumption that $\sK_{\bX}(w)$ is isotropic up to level $\cP^e$,  and that $e\geq k$. 
        We may conclude that $\sK_{\bX}(\omega)$ is isotropic up to level $\cP^{e+k}$ if and only if $ 2w^{\T}\bX^j\delta\equiv -C_j \bmod \cP^k$ for all $j \leq n-1$.

        It remains to show that every $\delta$ with $2w^{\T}\bX^j\delta\equiv -C_j \bmod \cP^k$ for $j \leq \operatorname{deg}(Q_e) + \operatorname{deg}(Q_k)-1$ satisfies the constraints for all larger values of $j$.  
        We proceed by induction on $j$. 
        Consider some $j \geq \operatorname{deg}(Q_e) $ with $ 2w^{\T}\bX^{j'}\delta\equiv -C_{j'} \bmod \cP^k$ for all $j'< j$. 
        
        Recall that $Q_e$ is monic with $ Q_e(\bX)w \equiv 0 \bmod \cP^{e}$. 
        It follows that there exist coefficients $q_i^{\bc{e}}  \in \sO_K$ with  
        \begin{equation}
        \textstyle 
           \bX^{\deg(Q_e)} \tilde{w} \equiv \sum_{i<\deg(Q_e)} q_i^{\bc{e}} \bX^{i}\tilde{w} + \pi^e r \bmod \cP^{e+k}.  \label{eq:HighNab} 
        \end{equation}
        Hence, for every $j\geq \operatorname{deg}(Q_e)$, using the definition $\tilde{w}^{\T} \bX^j \tilde{w} = \pi^e C_j$,  
       \begin{equation}
       \textstyle 
        C_{j} 
        \textstyle\equiv  \sum_{i<\deg(Q_e)}q_i^{\bc{e}} C_{j -\operatorname{deg}(Q_e) +i} + w^{\T} \bX^{ j - \operatorname{deg}(Q_e)} r \bmod \cP^k. 
       \end{equation}
       Similarly, reducing \eqref{eq:HighNab} modulo $\cP^k$ and using the induction hypothesis, 
       \begin{align}
        \textstyle
        2w^{\T}\bX^{j}\delta 
        & \textstyle\equiv \sum_{i<\deg(Q_e)}2q_i^{\bc{e}} w^{\T} \bX^{j-\operatorname{deg}(Q_e) + i} \delta \bmod \cP^{k}\\ 
        & \textstyle
        \equiv - \sum_{i<\deg(Q_e)}q_i^{\bc{e}}C_{j - \operatorname{deg}(Q_e) + i} \bmod \cP^{k}. \nonumber   
       \end{align}
       Thus, $2w^{\T}\bX^{j}\delta \equiv -C_j\bmod \cP^k$ for $j\geq \operatorname{deg}(Q_e)$ if and only if $w^{\T} \bX^{j- \operatorname{deg}(Q_e)}r \equiv 0 \bmod \cP^k$.
       It follows from $Q_k$ being monic with $Q_k(\bX)w \equiv 0 \bmod \cP^k$ that it is equivalent to impose the latter constraint only for $j\leq   \operatorname{deg}(Q_e) + \operatorname{deg}(Q_k)-1$. 
    \end{proof}
    \section{Numerical experiments}\label{sec: Numerical}      
    We finally apply Algorithm \ref{alg:orth_matrices} to random integer matrices with entries uniform on $[-100,100]\cap \bbZ$.
    To our knowledge, no numerical data is available in such settings in previous literature, due to the difficulty of exhaustive approaches.  
    Our experiments below use $M=10^4$ as the prime cutoff for Assumption \ref{ass: SmallPrimes}.
    \pagebreak[4]
    
    Table \ref{tab: RationalCospectral} concerns cospectrality over $\bbQ$ in varying dimensionality $n\leq 100$. 
    We observe that the probability of rational cospectrality peaks at approximately $40\%$  near $n\approx 7$. 
    This finding gives numerical evidence in favor of a conjecture by the author \cite[Conjecture 1.9]{vanwerde2026exact} that posits that the probability of rational cospectrality should be bounded away from one for integer matrices in fixed dimension.

    Rational cospectrality appears to be less frequent in high dimensions. 
    For example, no rationally cospectral mates were found in dimension $n=100$. 
    This is suggestive of an integer matrix variant of Haemers' conjecture concerning the spectral determinacy of high-dimensional random graphs \cite{van2003graphs,haemers2016almost}.

    We can also study the orthogonal matrix realizing cospectrality.
    The \emph{block size} $\lvert \bQ \rvert$ is the number of columns with strictly more than one nonzero entry.
    We identify the level ideal $\cL\subseteq \bbZ$ from \eqref{eq: Def_L} with the positive integer $\ell \in \bbZ_{\geq 1}$ satisfying $\cL = \ell \bbZ$.     
    Table \ref{tab: RationalCospectral} shows that the blocksize and level may both be relatively large when $n$ is small, but that it is typical to have small values for large $n$.
    Most rational cospectral mates for $n\geq 15$ had the smallest possible level $\ell =2$ and a small blocksize $\approx 4$.
    That small blocksizes are typical is partially explained by our proofs; recall \eqref{eq:MildCity}.    
    
        \begin{table}[h!]
  \centering
  \begin{tabular}{ccccc}
    \toprule
    $n$ & $\#$Cospectral in $\mathbb{Z}$ over $\mathbb{Q}$ & Avg. block size $\lvert \bQ \rvert$ & Avg. level $\ell$& Timeout \\
    \midrule
    2 & 74 & 2.000 & 9.350 & 0 \\
    3 & 246 & 2.838 & 16.269 & 0 \\
    4 & 315 & 3.594 & 8.933 & 0 \\
    5 & 360 & 4.120 & 9.904 & 0 \\
    6 & 383 & 4.878 & 11.790 & 0 \\
    7 & 407 & 5.492 & 103.945 & 0 \\
    8 & 379 & 6.167 & 18.131 & 0 \\
    9 & 389 & 6.265 & 12.765 & 0 \\
    10 & 343 & 6.507 & 9.222 & 0 \\
    11 & 253 & 6.492 & 20.571 & 0 \\
    12 & 182 & 6.113 & 5.028 & 0 \\
    13 & 132 & 5.705 & 3.496 & 0 \\
    14 & 84 & 4.869 & 2.768 & 0 \\
    15 & 54 & 4.817 & 2.413 & 2 \\
    16 & 23 & 4.580 & 2.540 & 1 \\
    17 & 20 & 4.286 & 2.048 & 0 \\
    18 & 6 & 4.000 & 2.000 & 1 \\
    19 & 14 & 4.235 & 2.000 & 1 \\
    20 & 4 & 4.286 & 2.143 & 1 \\
    25 & 0 & N/A & N/A & 0 \\
    50 & 0 & N/A & N/A & 2 \\
    100 & 0 & N/A & N/A & 0 \\
    \bottomrule
  \end{tabular}
  \caption{The number of integer matrices among $1000$ random samples where Algorithm \ref{alg:orth_matrices} with $M=10^4$ and $K = \bbQ$ finds a cospectral mate, as well as average the blocksizes and levels of the rational orthogonal matrices realizing the cospectrality.
  Computations that were inconclusive due to exceeding a time limit are reported under ``Timeout''. }
  \label{tab: RationalCospectral}
\end{table}

\pagebreak[4]

    We finally consider cospectrality over the quadratic fields $K \in  \{\bbQ(\sqrt{d}):d=2,3,5,6,7 \}$.
    Table \ref{tab:quadratic} provides data on the number of matrices that admit a nontrivial cospectral mate in $R^{n\times n}$ strictly over $K$ for $R\in \{\bbZ, \sO_K\}$. 
    Here, \emph{strictly} means that we consider only those mates that arise through conjugation with an orthogonal matrix over $K$ for which not all entries lie in the proper subfield $\bbQ$.
    
    We observe in Table \ref{tab:quadratic} that the quadratic fields produce a nontrivial number of additional cospectral mates, both in $\sO_K$ and $\bbZ$.
    Notably, the probability for $\bbZ$ depends sensitively on the dimensionality $n$. 
    For example, we find a substantial number of integer cospectral mates strictly over $\bbQ(\sqrt{2})$ only in even dimensions. 
    A similar phenomenon of sensitive dependence on the dimensionality is observed for $\bbQ(\sqrt{d})$ with $d>2$, but is not simply based on divisibility of $n$ by $d$.   
    
    \begin{table}[h!]
  \centering
  \begin{tabular}{lccccccccc}
    \toprule
     Dimension $n$& $2$ & $3$ & $4$ & $5$ & $6$ & $7$ & $8$ & $9$ & $10$ \\
    \midrule
    $\#\mathbb{Z}$ over $\mathbb{Q}(\sqrt{2})$ & 495 & 0 & 330 & 0 & 166 & 0 & 78 & 0 & 14 \\
    $\#\bbZ[\sqrt{2}]$ over $\mathbb{Q}(\sqrt{2})$ & 508 & 645 & 640 & 580 & 504 & 437 & 333 & 213 & 131 \\
    Timeout & 0 & 2 & 2 & 0 & 6 & 4 & 15 & 19 & 29 \\
     \midrule
    $\#\mathbb{Z}$ over $\mathbb{Q}(\sqrt{3})$ & 0 & 0 & 240 & 0 & 0 & 0 & 23 & 0 & 0 \\
    $\#\bbZ[\sqrt{3}]$ over $\mathbb{Q}(\sqrt{3})$ & 147 & 279 & 379 & 239 & 128 & 50 & 37 & 6 & 0 \\
    Timeout & 0 & 3 & 1 & 1 & 6 & 4 & 21 & 15 & 24\\
    \midrule 
    $\#\mathbb{Z}$ over $\mathbb{Q}(\sqrt{5})$ & 347 & 0 & 94 & 0 & 12 & 0 & 1 & 0 & 0 \\
    $\#\mathbb{Z}[\frac{\sqrt{5}+1}{2}]$ over $\mathbb{Q}(\sqrt{5})$ & 358 & 236 & 291 & 227 & 177 & 95 & 53 & 5 & 4 \\
    Timeout & 0 & 0 & 1 & 0 & 9 & 5 & 24 & 17 & 23 \\
    \midrule 
    $\#\mathbb{Z}$ over $\mathbb{Q}(\sqrt{6})$ & 0 & 0 & 67 & 0 & 0 & 0 & 3 & 0 & 0 \\
    $\#\mathbb{Z}[\sqrt{6}]$ over $\mathbb{Q}(\sqrt{6})$ & 36 & 103 & 108 & 27 & 7 & 0 & 3 & 0 & 0 \\
    Timeout & 0 & 1 & 2 & 2 & 3 & 5 & 17 & 16 & 26 \\
    \midrule 
    $\#\mathbb{Z}$ over $\mathbb{Q}(\sqrt{7})$ & 0 & 0 & 47 & 0 & 0 & 0 & 0 & 0 & 0 \\
    $\#\mathbb{Z}[\sqrt{7}]$ over $\mathbb{Q}(\sqrt{7})$ & 39 & 35 & 59 & 15 & 0 & 0 & 0 & 0 & 0 \\
    Timeout & 0 & 1 & 3 & 2 & 4 & 2 & 11 & 13 & 27 \\
    \bottomrule
  \end{tabular}
  \caption{The number of integer matrices in $1000$ samples where Algorithm \ref{alg:orth_matrices} with $M=10^4$ finds a cospectral mate with entries in $R\in \{\bbZ,\sO_K \}$ strictly over $K =\bbQ(\sqrt{d})$.
  Computations that were inconclusive due to exceeding a time limit are reported under ``Timeout''.}
  \label{tab:quadratic}
\end{table}

Let us finally note that Algorithm \ref{alg:orth_matrices} is also applicable to simple graphs. 
However, phenomena related to Remark \ref{rem: SimpleDisc} make it less efficient in high dimensions ($n\gg 20$). 
Potential refinements of the algorithm to simple graphs are an interesting problem, which we intend to pursue in future work.  
We refer to \cite{wang2025haemers,van2026exact} for numerical data on generalized cospectrality in simple graphs. 
Numerical data for graphs with loops and signed graphs is given in Appendix \ref{apx: NumericalSigned}.
\pagebreak[4]
     \addtocontents{toc}{\SkipTocEntry}
  \subsection*{Acknowledgements}
    I thank Nikita Lvov and Nils Van de Berg for discussions related to the topic of this paper. 
    
    Funded by the Deutsche Forschungsgemeinschaft (DFG, German Research Foundation) under Germany's Excellence Strategy EXC 2044/2 –390685587, Mathematics Münster: Dynamics–Geometry–Structure.
     \addtocontents{toc}{\SkipTocEntry}
 
     \subsection*{Tool and computational resource disclosure}
     The author developed the main mathematical content and the writing. 
     In the preparation of this manuscript, large language models (LLMs) were used as a supporting tool for language refinement or idea exploration, never to autonomously generate text.
     For the codebase, LLMs were used to optimize the implementation, as discussed below.
    
    The initial implementation (v0.0.1 on GitHub) was written manually, with LLM use (Google Gemini Pro 3.1) being restricted to \texttt{SageMath} syntax searches and local code fragment suggestions. 
    The initial implementation was however relatively slow, which we suspect to be primarily due to reliance on a generic \texttt{SageMath} quotient ring constructor whose arithmetic is not optimized for the special case of (quotients of) rings of integers.
    There were also some bugs. 

Iterations to improve the code efficiency and quality were hereafter done with more intensive LLM use, mostly with Anthropic Claude Opus 4.8.  
To ensure that the code remained maintainable and understood, a two-stage process was used. 
First, a LLM was used to aid in drafting design specifications that addressed issues in the code, which the author then edited. 
Next, the specifications were implemented using a separate LLM instance.

The resulting code was reviewed by the author and tested on exhaustively generated examples, including all cospectral $3\times 3$ matrices with $\sum_{i\leq j} \bX_{i,j}^{2} \leq 25$. 
The algorithm was run over the field $K$ required to realize the cospectrality, which was determined using the orthogonal matrices from the pair's eigendecomposition. 
The degree of the required extension $K/\bbQ$ ranged from $1$ to $4$.

The experiments in Section \ref{sec: Numerical} used version 0.0.3 of the implementation.
The source code is available at \url{https://github.com/Alexander-Van-Werde/cospectral}.

    \appendix
    
    \section{Matrices with discriminant \texorpdfstring{$\pm 1$}{+/-1} --- Proof of Proposition \ref{prop: Discpm1}}\label{sec: MatrixDiscpm1}
    Recall from Remark \ref{rem: DiscPoly} that the discriminant of a monic polynomial $f\in \bbZ[x]$ with roots $\lambda_1,\ldots,\lambda_n\in \bbC$ is given by $\Delta_f = \prod_{i<j}  (\lambda_i -\lambda_j)^2$.
    \begin{lemma}\label{lem: IrreducibleDegree1}
        Let $f\in \bbZ[x]$ be a monic polynomial with $\Delta_f = \pm 1$. 
        Then, all irreducible factors of $f$ over $\bbZ[x]$ have degree one. 
    \end{lemma}
    \begin{proof}
        Suppose to the contrary that $f(x) = a(x) b(x)$ for some monic irreducible $a(x)\in \bbZ[x]$ with $\deg(a) \geq 2$. 
        Then, $\Delta_f = \Delta_a \Delta_b R(a,b)^2$ with $R(a,b) = \prod_{i}\prod_j (\alpha_i -\beta_j)$ for $\alpha_i$ (resp. $\beta_j$) running over the roots of $a(x)$ (resp. $b(x)$).
        Here, $R(a,b)$ is the \emph{resultant} of $a,b$ is again an integer; see \eg \cite[Chapter IV, Section 8]{lang2012algebra}.
        It hence remains to show that $\lvert \Delta_a \rvert \geq 2$.
      
        Let $\alpha\in \bbC$ be an arbitrary root of $a$, let $K \de \bbQ(\alpha)$, and denote $k \de \deg(a)$. 
        Recall that the \emph{discriminant} of $k_1,\ldots,k_m\in K$ is $\operatorname{disc}(k_1,\ldots,k_k) \de \det([\sigma_i (k_j) ])^2$ with $\sigma_i:K\to \bbC$ the field embeddings \cite[p.18]{marcus1977number}.
        Recall the classical fact that, 
        \begin{equation}
            \operatorname{disc}(1,\alpha, \ldots,\alpha^{k-1}) =  \prod_{i < j} \bigl( \sigma_i(\alpha) - \sigma_j(\alpha) )^2  = \pm \Delta_a , \label{eq:NewOak}
        \end{equation}
        which can be proved by considering a Vandermonde determinant \cite[p.19]{marcus1977number}.
        Let $\operatorname{disc}(\sO_K)$ be the discriminant of an arbitrary $\bbZ$-module basis $\gamma_1,\ldots,\gamma_k \in \sO_K$.
        Then, $(1,\alpha, \ldots,\alpha^{k-1})^{\T} = \bM (\gamma_1,\ldots,\gamma_k)^{\T}$ for some $\bM \in \bbZ^{k\times k}$ and hence, 
        \begin{equation}
            \operatorname{disc}(1,\alpha, \ldots, \alpha^{k-1}) = \det(\bM)^2 \operatorname{disc}(\sO_K).\label{eq:AlertQuip} 
        \end{equation}
        Note that $K\neq \bbQ$ since $a$ is monic and irreducible over $\bbZ[x]$ and hence also over $\bbQ$ by Gauss's lemma \cite[Theorem 1, p.10]{marcus1977number}. 
        It however holds that $\lvert\operatorname{disc}(\sO_K)\rvert \geq 2$ for any non-trivial number field $K \neq \bbQ$ \cite[p.96]{marcus1977number}.    
        The result hence follows from \eqref{eq:NewOak} and \eqref{eq:AlertQuip} since $\det(\bM) \in\bbZ$ and $\Delta_a \neq 0$.  
    \end{proof}
    
    \begin{lemma}\label{lem: Degree2}
        Let $f\in \bbZ[x]$ be a monic polynomial for which all roots are real and $\deg(f)>1$.
        Then, $\Delta_f = \pm 1$ implies that $\deg(f) = 2$.
    \end{lemma}
    \begin{proof}
        Lemma \ref{lem: IrreducibleDegree1} implies that $f$ factors in linear factors, and the formula $\Delta_f =  \prod_{i<j}  (\lambda_i -\lambda_j)^2$ implies that all roots of $f$ have to satisfy $\lvert \lambda_i - \lambda_j \rvert = 1$.
        By the structure of the real line, this is impossible if $n\geq 3$.  
    \end{proof}

    \begin{proof}[Proof of Proposition \ref{prop: Discpm1}]
        It follows from Lemma \ref{lem: Degree2} that $\lvert \Delta_{\bX} \rvert=1$ for a symmetric integer matrix $\bX\in \bbZ^{n\times n}$ with $n>1$ implies then $n=2$.
        In this case,   
        \begin{equation}
            \bX = \begin{pmatrix}
                a & b \\ 
                b & c
            \end{pmatrix}.
        \end{equation}
        Then, $
            \phi_\bX(x) = x^2 - (a+c)x - b^2 + ac$ and  
        $
           \lvert  \Delta_{\bX} \rvert  = (a-c)^2 + 4b^2  
        $. 
        Hence, $\lvert \Delta_{\bX} \rvert=1$ if and only if $\lvert a-c \rvert=1$ and $b=0$. 
        \end{proof}

    \section{Numerical data for signed graphs and graphs with loops}\label{apx: NumericalSigned}
    We here repeat the experiments concerning rational cospectrality in Table \ref{tab: RationalCospectral}, but now consider the adjacency matrices of graphs with loops and signed graphs.

     \addtocontents{toc}{\SkipTocEntry}
    \subsection{Signed graphs}
   We consider a random symmetric matrix with the entries  $\{\bX_{i,j}: i\leq j \}$ uniform and independent from $\{-1,0,1\}$. 
    Table \ref{tab: RationalCospectralSignedZZ} reports on the number of samples where Algorithm \ref{alg:orth_matrices} finds a cospectral mate with entries in $\bbZ$, while Table \ref{tab: RationalCospectralSigned} restricts to cospectral mates that again have entries in $\{-1,0,1 \}$. 
    \begin{table}[h!]
  \centering
  \footnotesize
  \begin{tabular}{ccccc}
    \toprule
   $n$ & $\#$Cospectral in $\mathbb{Z}$ over $\mathbb{Q}$ & Avg. block size $\lvert \bQ \rvert$ & Avg. level $\ell$& Inconclusive \\
    \midrule
    5 & 150 & 4.205 & 4.586 & 51 \\
    6 & 270 & 4.914 & 7.694 & 32 \\
    7 & 369 & 5.690 & 22.416 & 16 \\
    8 & 400 & 6.807 & 35.120 & 2 \\
    9 & 404 & 6.396 & 37.608 & 1 \\
    10 & 322 & 6.730 & 36.746 & 0 \\
    11 & 268 & 6.014 & 6.378 & 1 \\
    12 & 219 & 6.084 & 7.675 & 0 \\
    13 & 135 & 5.825 & 3.666 & 1 \\
    14 & 94 & 5.588 & 2.748 & 1 \\
    15 & 58 & 4.558 & 2.274 & 1 \\
    16 & 31 & 4.407 & 2.169 & 0 \\
    17 & 28 & 4.654 & 2.231 & 0 \\
    18 & 19 & 4.171 & 2.086 & 1 \\
    19 & 7 & 4.308 & 2.154 & 0 \\
    20 & 10 & 4.211 & 2.158 & 3 \\
    25 & 0 & N/A & N/A & 1 \\
    50 & 0 & N/A & N/A & 2 \\
    \bottomrule
  \end{tabular} 
  \caption{The number of signed graphs among $1000$ random samples  where Algorithm \ref{alg:orth_matrices} with $M=10^4$ and $K = \bbQ$ finds an integer cospectral mate to the adjacency matrix, and orthogonal matrix statistics. 
  Computations that exceeding a time limit or could not be initiated due to a vanishing discriminant are reported under ``Inconclusive''. }
  \label{tab: RationalCospectralSignedZZ}
\end{table}

    \begin{table}[h!]
  \centering
  \footnotesize
  \begin{tabular}{ccccc}
    \toprule
   $n$&$\#$Cospectral in $\{-1,0,1\}$ over $\mathbb{Q}$&Avg. block size $\lvert \bQ \rvert$&Avg. level $\ell$&Inconclusive\\
    \midrule
    5 & 47 & 3.718 & 3.254 & 51 \\
    6 & 125 & 4.677 & 8.463 & 32 \\
    7 & 164 & 5.183 & 24.342 & 16 \\
    8 & 149 & 5.676 & 25.091 & 2 \\
    9 & 116 & 5.322 & 33.441 & 1 \\
    10 & 54 & 4.957 & 19.565 & 0 \\
    11 & 47 & 4.231 & 2.750 & 1 \\
    12 & 25 & 4.531 & 2.156 & 0 \\
    13 & 8 & 4.750 & 2.375 & 1 \\
    14 & 8 & 4.600 & 2.100 & 1 \\
    15 & 4 & 3.750 & 2.250 & 1 \\
    16 & 3 & 4.000 & 2.000 & 0 \\
    17 & 2 & 4.000 & 2.000 & 0 \\
    18 & 0 & N/A & N/A & 1 \\
    19 & 0 & N/A & N/A & 0 \\
    20 & 0 & N/A & N/A & 3 \\
    25 & 0 & N/A & N/A & 1 \\
    50 & 0 & N/A & N/A & 2 \\ 
    \bottomrule
  \end{tabular} 
  \caption{The number of signed graphs among $1000$ random samples  where Algorithm \ref{alg:orth_matrices} finds a $\{-1,0,1 \}^{n\times n}$-valued cospectral mate. 
  The remainder of the setup is as in the caption of Table \ref{tab: RationalCospectralSignedZZ}. }
  \label{tab: RationalCospectralSigned}
\end{table}
\addtocontents{toc}{\SkipTocEntry}
\subsection{Unsigned graphs (possibly with loops)}

   We consider a random symmetric matrix with the entries  $\{\bX_{i,j}: i\leq j \}$ uniform and independent from $\{0,1\}$. 
   Table \ref{tab: RationalCospectralLoopyZZ} reports on the number of samples where Algorithm \ref{alg:orth_matrices} finds a cospectral mate with entries in $\bbZ$. 
\begin{table}[h!]
  \centering
  \footnotesize
  \begin{tabular}{ccccc}
    \toprule
    $n$ & $\#$Cospectral in $\mathbb{Z}$ over $\mathbb{Q}$ & Avg. block size $\lvert \bQ \rvert$ & Avg. level $\ell$& Inconclusive \\
    \midrule
    5 & 53 & 4.344 & 3.958 & 211 \\
    6 & 149 & 4.920 & 5.243 & 155 \\
    7 & 205 & 5.669 & 11.850 & 102 \\
    8 & 259 & 6.045 & 11.853 & 67 \\
    9 & 279 & 6.732 & 137.320 & 47 \\
    10 & 266 & 6.499 & 317.561 & 28 \\
    11 & 235 & 6.676 & 62.949 & 13 \\
    12 & 157 & 6.735 & 31.331 & 9 \\
    13 & 124 & 6.307 & 52.952 & 3 \\
    14 & 75 & 7.265 & 78.273 & 4 \\
    15 & 56 & 4.850 & 4.009 & 3 \\
    16 & 29 & 4.724 & 3.931 & 1 \\
    17 & 22 & 4.238 & 2.119 & 1 \\
    18 & 10 & 5.000 & 12.353 & 2 \\
    19 & 6 & 4.000 & 2.000 & 1 \\
    20 & 3 & 4.000 & 2.000 & 2 \\
    25 & 1 & 4.000 & 2.000 & 0 \\
    50 & 0 & N/A & N/A & 2 \\
    \bottomrule
  \end{tabular}
  \caption{The number of graphs (possibly with loops) among $1000$ random samples  where Algorithm \ref{alg:orth_matrices} finds an integer cospectral mate.
  The remainder of the setup is as in the caption of Table \ref{tab: RationalCospectralSignedZZ}.}
  \label{tab: RationalCospectralLoopyZZ}
\end{table}

Table \ref{tab: RationalCospectralSigned} restricts to cospectral mates in that again correspond to unsigned graphs (possibly with loops). 
Finally, Table \ref{tab: RationalCospectralLoopy} considers cospectrality to signed graphs.

    \begin{table}[h!]
  \centering
  \footnotesize
  \begin{tabular}{ccccc}
    \toprule
   $n$&$\#$Cospectral in $\{0,1\}$ over $\mathbb{Q}$&Avg. block size $\lvert \bQ \rvert$&Avg. level $\ell$&Inconclusive\\
    \midrule
     5 & 10 & 4.286 & 3.143 & 211 \\
    6 & 50 & 4.494 & 4.071 & 155 \\
    7 & 76 & 5.170 & 14.717 & 102 \\
    8 & 123 & 5.560 & 11.632 & 67 \\
    9 & 142 & 5.567 & 93.372 & 47 \\
    10 & 145 & 5.931 & 22.121 & 28 \\
    11 & 149 & 5.886 & 57.270 & 13 \\
    12 & 97 & 5.725 & 11.775 & 9 \\
    13 & 83 & 5.443 & 6.402 & 3 \\
    14 & 59 & 6.656 & 77.021 & 4 \\
    15 & 48 & 4.909 & 5.255 & 3 \\
    16 & 24 & 4.480 & 4.120 & 1 \\
    17 & 22 & 4.435 & 2.217 & 1 \\
    18 & 8 & 5.889 & 21.556 & 2 \\
    19 & 6 & 4.000 & 2.000 & 1 \\
    20 & 3 & 4.000 & 2.000 & 2 \\
    25 & 1 & 4.000 & 2.000 & 0 \\
    50 & 0 & N/A & N/A & 2 \\
    \bottomrule
  \end{tabular} 
  \caption{The number of graphs (possibly with loops) among $1000$ random samples  where Algorithm \ref{alg:orth_matrices} finds a $\{0,1 \}^{n\times n}$-valued cospectral mate. 
  The remainder of the setup is as in the caption of Table \ref{tab: RationalCospectralSignedZZ}. }
  \label{tab: RationalCospectralSigned}
\end{table}

    \begin{table}[h!]
  \centering
  \footnotesize
  \begin{tabular}{ccccc}
    \toprule
   $n$&$\#$Cospectral in $\{-1,0,1\}$ over $\mathbb{Q}$&Avg. block size $\lvert \bQ \rvert$&Avg. level $\ell$&Inconclusive\\
    \midrule 
     5 & 13 & 4.278 & 3.111 & 211 \\
    6 & 61 & 4.565 & 6.037 & 155 \\
    7 & 91 & 5.329 & 13.193 & 102 \\
    8 & 133 & 5.714 & 12.028 & 67 \\
    9 & 154 & 5.765 & 73.085 & 47 \\
    10 & 153 & 6.033 & 19.811 & 28 \\
    11 & 160 & 5.966 & 60.461 & 13 \\
    12 & 104 & 5.840 & 11.165 & 9 \\
    13 & 87 & 5.435 & 7.344 & 3 \\
    14 & 61 & 7.256 & 79.288 & 4 \\
    15 & 48 & 4.847 & 5.034 & 3 \\
    16 & 25 & 4.500 & 3.929 & 1 \\
    17 & 22 & 4.417 & 2.208 & 1 \\
    18 & 8 & 5.889 & 21.556 & 2 \\
    19 & 6 & 4.000 & 2.000 & 1 \\
    20 & 3 & 4.000 & 2.000 & 2 \\
    25 & 1 & 4.000 & 2.000 & 0 \\
    50 & 0 & N/A & N/A & 2 \\
    \bottomrule
  \end{tabular} 
  \caption{The number of graphs (possibly with loops) among $1000$ random samples  where Algorithm \ref{alg:orth_matrices} finds a $\{-1,0,1 \}^{n\times n}$-valued cospectral mate. 
  The remainder of the setup is as in the caption of Table \ref{tab: RationalCospectralSignedZZ}. }
  \label{tab: RationalCospectralLoopy}
\end{table}

\end{document}